\documentclass[11pt, oneside]{amsart}
\usepackage{pgfplots}
\definecolor{grayy}{rgb}{0.5,0.5,0.5}

\definecolor{greenn}{rgb}{0,0.5,0}

\newcommand{\abs}[1]{\left| #1 \right|}

\newtheorem{theorem}{Theorem}
\newtheorem*{theorem*}{Theorem}
\newtheorem{lemma}{Lemma}
\newtheorem*{lemma*}{Lemma}
\newtheorem{proposition}{Proposition}
\newtheorem*{proposition*}{Proposition}
\newtheorem{corollary}{Corollary}
\newtheorem*{corollary*}{Corollary}
\newtheorem{proc}{Procedure}

\newtheorem*{example*}{Example}

\newtheorem{definition}{Definition}
\newtheorem*{definition*}{Definition}

\theoremstyle{remark}
\newtheorem{remark}{Remark}
\newtheorem*{remark*}{Remark}

\newcommand\mshift[2]{{}^{#2}\!{#1}}
\newcommand{\vsg}[1]{{<}#1{>}}
\newcommand{\Y}{Y_0,Y_1,\ldots,Y_n}

\DeclareMathOperator{\ord}{ord}
\DeclareMathOperator{\supp}{supp}
\DeclareMathOperator{\Coeff}{Coeff}
\DeclareMathOperator{\coeff}{Coeff}
\DeclareMathOperator{\NIC}{NIC}
\DeclareMathOperator{\RatRk}{rat.rk}

\def\vq{\vert q \vert}

\def\ud{\underline{d}}

\author{J. Cano}
\address{Universidad de Valladolid, Spain}
\email{jcano@agt.uva.es}
\author{P. Fortuny Ayuso}
\address{Universidad de Oviedo, Spain}
\email{fortunypedro@uniovi.es}
\date{\today}
\title[Newton-Puiseux Polygon of non-linear $q-$Difference Equations]{Power Series Solutions of Non-Linear $q$-Difference
  Equations and the Newton-Puiseux Polygon}

\thanks{Partially supported by the Ministerio de Ciencia e Innovación (Spain), Proyect id. PID2019-105621GB-I00}
\usepackage{xifthen}
\newboolean{pedro}
\setboolean{pedro}{true}
\newcommand{\pedro}{}
\newboolean{jose}
\setboolean{jose}{true}
\newcommand{\jose}{\ifthenelse{\boolean{jose}}{\color{blue}
    \setboolean{jose}{false}}{\color{black}\setboolean{jose}{true}}}

\begin{document}
\begin{abstract}
  Adapting the Newton-Puiseux Polygon process to nonlinear 
  $q-$di\-ffe\-ren\-ce
  equations of any order and degree, we compute their power series
  solutions, study the properties of the set of exponents of  the
  solutions and give
  a bound for their $q-$Gevrey
  order in terms of the order of the original equation.
\end{abstract}
\maketitle

\section{Introduction}
The Newton Polygon construction for solving equations in terms of
power series and its generalization by Puiseux has been successfully
used 
countless  
times both in the algebraic \cite{Puiseux1}, \cite{Puiseux2}
\cite{Hironaka-polyhedra} and in
the differential contexts
\cite{Fine}, \cite[Ch. V]{Forsyth}, \cite{Ince}, \cite{Grigoriev-Singer}, \cite{CanoJ},
\cite{Cano-Fortuny-Asterisque}, \cite{DellaDoraJung}, \cite{vDH-LN}
(this is just a biased and brief sample, 
see also 
\cite{Christensen-1996} and \cite[Sec. 29]{Enriques} for an
interesting detailed historical narrative).
We extend
its use to $q-$difference equations.

Although this construction is primarily intended to give a method 
for computing formal power series solutions, 
we will use it for proving the $q$-analog of some results
concerning the nature of power series solutions of non linear
differential equations. Namely, we show properties about the growth of
the coefficients of a power series solution (Maillet's theorem) and
about the set of exponents of a generalized power series solution.

The method allows us, first of all, to show that the set of exponents
of \pedro{}power series solutions with well-ordered exponents in
$\mathbb{R}$\pedro{} of a formal $q$-difference equation is
\pedro{}included in the translation by a constant of a finitely generated semigroup over
$\mathbb{Z}_{\geq 0}$\pedro{} (in particular, it has finite rational
rank and if the exponents are all rational, then their denominators
are bounded).  This mirrors the results of D. Y. Grigoriev and
M. Singer in \cite{Grigoriev-Singer} for differential equations.  When the
$q$-difference equation is of first order and first degree, we give a
bound for this rational rank (see Theorem~\ref{th:bound-rati-rank} for
a precise statement). We also study properties related to what we call
``finite determination'' (Definition~\ref{de:finite-determination}) of
the coefficients of the solutions. This is one of the places in which
the case $|q|=1$ is essentially different from the general case. For
$|q|\neq 1$, we prove the finite determination of the coefficients.

Maillet's theorem~\cite{Maillet} is a classical results about the
growth of the 
coefficients $a_i$ of a formal power series solution of a (non-linear)
differential equation: it states that
$|a_i|\leq i!^{s}\,R^{i}$, for some 
constants $R$ and $s$. Among the different proofs 
(for  instance
\cite{Maillet,Mahler,Gerard}), 
B. Malgrange's~\cite{Malgrange-sur-maillet}  includes
a precise  bound for~$s$. This bound is 
optimal except for one case: when the linearized operator along the
solution has a 
regular singularity and the solution is a ``non-regular solution'',
for which any $s>0$ works
(see the last remark in Malgrange's paper);
we shall refer to it as the (RS-N) case.  
In~\cite{CanoJ}, the Newton Polygon method allows the author to prove
Maillet's result and to show 
convergence (i.e. $s=0$) in the (RS-N) case.

The first studies on convergence of solutions of non-linear
$q$-difference equations are due to B\'ezivin 
\cite{Bezivin-q-differences},
\cite{Bezivin-convergence-certain-eq-fun} and
\cite{Bezivin-eq-q-diff-p-adiques}.
The $q$-analog of Maillet's theorem states that when
$|q|>1$, 
a formal power 
series solution of a $q$-difference equation 
with analytic coefficients is $q$-Gevrey of some order $s$ (see
Definition~\ref{def:q-gevrey-order}). 
Zhang~\cite{Zhang} proves this  adapting Malgrange's proof to the case
of $q$-difference--differential 
convergent equations. In this paper,
the adaptation of the Newton Polygon 
to $q$-difference equations allow us to give a new proof of 
the $q$-analogue of Maillet's theorem and to extend it to the 
$q$-Gevrey non-convergent case. 
The bounds obtained for convergent equations match 
Zhang's in general 
 and are more accurate in the (RS-N) case. 
However, we cannot prove convergence in this case unlike for
differential equations.

{} The first version of this paper was uploaded to the arXiv as
\cite{Cano-Fortuny-Arxiv-2012} in 2012. Parts of the second section
became a chapter of \cite{Barbe-Cano-Fortuny-McCormick-arxiv}, a joint
work with Ph. Barbe and W. McCormick dealing with solutions of
\emph{algebraic} $q$-difference equations. In that joint book, some
results concerning the asymptotic behavior of solutions are provided,
but the ones here are previous, more general (power series) and
stronger (due to the specific technique). However, we remark that in
\cite{Barbe-Cano-Fortuny-McCormick-arxiv} the topics are broader:
analytic, entire and formal solutions, the radius of convergence,
conditions describing the possible poles of analytic solutions,
associated objects which provide information on the solution
(Borel-type transforms), and many exhaustive examples, among which:
the colored Jones equation for the figure $8$ knot, the $q$-Painlevé I
equation, and other combinatorial equations. Thus, the present paper
is transverse to the book, \pedro{}and the Newton Polygon method
applied to $q$-difference equations (which appears in both) was first
used in this work.\pedro{} {}

{}We note, also,{} that the ``Newton Polygon'' construction used in the case of
linear operators by Adams \cite{Adams}, 
Ramis \cite{Ramis-q-diff-toulouse}, 
Sauloy \cite{Sauloy-filtration-Fourier} and others is different from
the one presented here. In the linear case, the Newton Polygon is used to
find local invariants of the operator while our Newton Polygon
 is constructed with the aim of looking for formal power
series solutions. In Section~\ref{sec:qGevreyOrder} 
we describe the relation between Adams' Newton Polygon
and Zhang's bounds.  Adams' construction is also used
in \cite{Li-Zhang-2011}
to give conditions for the convergence
of the solution(s) of analytic nonlinear $q$-difference equations.

For the reader's convenience, we include a final section with a
detailed working example describing most of the constructions and the
evolution of the Newton Polygon as one computes the successive terms
of a solution.

\section{The Newton-Puiseux Polygon process for $q$-difference equations}%
Let $q$ be a nonzero complex number. For $j\in \mathbb{Z}$, 
let us denote by $\sigma^j$ the automorphism of
the ring ${\mathbb C}[[x]]$ of formal power series in one
variable  given by \pedro{}$\sigma^j(y)(x)=y(q^jx)$\pedro{}, that is,
\begin{displaymath}
  \sigma^j(\sum_{i=0}^\infty a_i\,x^i) = 
\sum_{i=0}^\infty q^{i\,j}\,a_i\,x^i.
\end{displaymath}


Let $ P(x,Y_{0},Y_{1},\dots,Y_{n})\in
\mathbb{C}[[x,Y_{0},\dots,Y_{n}]]$ be a formal power series.  
For $y\in \mathbb{C}[[x]]$, with
$\ord_x(y)>0$, the expression
$P(x,y,\sigma^1(y),\dots,\sigma^n(y))$ 
is a well-defined element of
$\mathbb{C}[[x]]$ that we will be denoted by $P[y]$.  
We associate to $P(x,Y_{0},Y_{1},\dots,Y_{n})$ 
the \emph{$q-$difference equation} 
\begin{equation}
  \label{eq:q-difference-equation-intro}
  P(x,y,\sigma^1(y),\dots,\sigma^n(y))= 0.
\end{equation}

We will look for solutions of
equation~(\ref{eq:q-difference-equation-intro}) 
as formal power series with real exponents. 
We restrict ourselves to 
the Hahn field
$\mathbb{C}((x^{\mathbb{R}}))$   of 
generalized power series, that is, formal power
series of the form   
$\sum_{\gamma\in \mathbb{R}}
c_{\gamma}x^{\gamma}$ 
whose support $\{\gamma\mid
c_{\gamma}\neq 0\}$ is a  well-ordered subset of
$\mathbb{R}$ and $c_\gamma\in \mathbb{C}$.  
Hahn fields  were essentially introduced in  
\cite{Hahn-series}; see \cite{Ribenboim} for a detailed proof of the
ring structure and 
\cite{vdH-generalized-power-series} for a modern study in the
context of functional equations. We fix a determination of the
logarithm and extend the automorphism $\sigma$ to
$\mathbb{C}((x^{\mathbb{R}}))$ by setting
\begin{displaymath}
  \sigma(\sum_{\gamma\in \mathbb{R}}
c_{\gamma}\,x^{\gamma}) = 
\sum_{\gamma\in \mathbb{R}}
q^{\gamma}\,c_{\gamma}\,x^{\gamma}.
\end{displaymath}
For $y\in \mathbb{C}((x^{\mathbb{R}}))$,  
its order $\ord(y)$
is the minimum of its support if $y\neq 0$ and
$\ord(0)=\infty$.  
In subsection \ref{sec:compositio-meaning}, we shall see that  
if $\ord(y)>0$ then the expression
$P(x,y,\sigma^1(y),\dots,\sigma^n(y))$ 
is a well-defined element of $\mathbb{C}((x^{\mathbb{R}}))$, hence
equation~(\ref{eq:q-difference-equation-intro}) makes sense in our setting. 

Although we look for solutions in the Hahn field, 
their support has some
finiteness properties, as in the case for differential equations.
We say that $y\in \mathbb{C}((x^{\mathbb{R}}))$ is a
grid-based series if there exists $\gamma_0\in \mathbb{R}$ and
a finitely generated semigroup $\Gamma\subseteq \mathbb{R}_{\geq 0}$
 such that the support of $y$ is contained in 
$\gamma_0+\Gamma$. 
Puiseux series are the particular case of grid-based series in which
$\gamma_0\in \mathbb{Q}$ and $\Gamma\subseteq \mathbb{Q}$.
Puiseux series 
and grid-based series form
subfieds of the Hahn field denoted respectively by $\mathbb{C}((x^\mathbb{Q}))^g$ and 
$\mathbb{C}((x^{\mathbb{R}}))^g$. We have
\begin{displaymath}
  \mathbb{C}[[x]]\subseteq \mathbb{C}((x^\mathbb{Q}))^g
  \subseteq \mathbb{C}((x^{\mathbb{R}}))^g \subseteq 
 \mathbb{C}((x^{\mathbb{R}})).
\end{displaymath}

If equation~(\ref{eq:q-difference-equation-intro}) is algebraic, i.e.
of the form $P(x,y)=0$, then  by Puiseux's Theorem  
all its formal power series solutions are of Puiseux type.
This is no longer true if instead of ${\mathbb C}$, 
the base field is of positive characteristic,
as the following example (due essentially to Ostrowski) shows: 
the equation $-y^p+x\,y+x=0$ over the field $\mathbb{Z}/p\mathbb{Z}$
has as solution the generalized power series 
$y=\sum_{i=1}^{\infty}x^{\mu_i}$ with 
$\mu_i={(p^i-1)/(p^{i+1}-p^{i})}$. Notice that the exponents are
rational but they do not have a common denominator and moreover 
$\mu_1<\mu_2<\cdots<1/(p-1)$ so that they do not even go to infinity.  
Hence $y$ is neither a Puiseux series
nor a grid-based series.

As in the case of differential equations, the number of generalized
power series solutions of a  
given equation~(\ref{eq:q-difference-equation-intro}) is not necessary
finite, neither all of its solutios are of Puiseux type. For instance,
the $q$-difference 
equation $Y_0\,Y_2-Y_1^2=0$ has $c\,x^{\mu}$ as solutions for any
$c\in \mathbb{C}$ and $\mu\in \mathbb{R}$.

\subsection{The Newton Polygon}
\label{sec:newt-puis-polyg}%
Let $\mathcal{R}=\mathbb{C}[[x^{\mathbb{R}_{\geq 0}}]]$ be the
ring of generalized power series with non-negative order. 
For a finitely generated semigroup of 
$\Gamma\subset\mathbb{R}_{\geq  0}$, the ring
$\mathbb{C}[[x^{\Gamma}]]$
formed by those generalized power series with support contained in
$\Gamma$
is denoted by $\mathcal{R}_\Gamma$.
Let $P\in
\mathcal{R}[[Y_0,Y_1,\ldots,Y_n]]$ 
be a nonzero formal power series in $n+1$
variables over~$\mathcal{R}$.
For $\rho=(\rho_0,\rho_1,\ldots,\rho_n)\in\mathbb{N}^{n+1}$, 
we shall write
$Y^{\rho}=Y_{0}^{\rho_{0}}\cdot Y_1^{\rho_1}\cdots Y_{n}^{\rho_{n}}$;
we shall also write $\mathcal{R}[[Y]]$ instead of $\mathcal{R}[[\Y]]$.
The coefficient of $Y^{\rho}$ in $P$ will be denoted
$P_\rho(x)\in
\mathcal{R}$ and, 
for $\alpha\in \mathbb{R}$, 
the coefficient of $x^{\alpha}$ in
$P_\rho(x)$ will be denoted 
$P_{\alpha,\rho}\in \mathbb{C}$.
\pedro{}
Notice that, as $P\in \mathcal{R}[[Y_0,Y_{1},\ldots, Y_{n}]]$, each coefficient $P_{\rho}(x)$ belongs to $\mathcal{R}$, which means that $P_{\rho}(x)$ is a power series with well-ordered support contained in $\mathbb{R}_{\geq 0}$. Thus, we can write:
\pedro{}
$$
P=
\sum_{\rho\in \mathbb{N}^{n+1}}
P_\rho(x)\,Y^{\rho},\quad \text{and}\quad
P_\rho(x)=
\sum_{\alpha\in\Gamma_\rho} P_{\alpha,\rho}\,
x^{\alpha},
$$
where for each $\rho$, $\Gamma_\rho$ is a well-ordered
subset of $\mathbb{R}_{\geq 0}$ \pedro{}(in general, the $\Gamma_{\rho}$ will all be different)\pedro{}. 
We associate to $P$ 
its  \emph{cloud of points} $\mathcal{C}(P)$: the set of 
points \pedro{}$(\alpha,|\rho|)\in {\mathbb R}^2$ with
$|\rho|=\rho_{0}+\rho_{1}+\dots+\rho_{n}$, for all $(\alpha,\rho)$ such that $P_{\alpha,\rho}\neq
0$\pedro{}.

The \emph{Newton Polygon}  $\mathcal{N}(P)$ of $P$
is the convex hull of
\begin{displaymath}
  \bar{\mathcal{C}}(P)=
  \{(\alpha+r,|\rho|)\mid (\alpha,|\rho|)\in \mathcal{C}(P),\,\, r\in
  \mathbb{R}_{\geq 0}\}.
\end{displaymath}
A \pedro{}\emph{supporting line} $L$ of $\mathcal{N}(P)$ is a line
such that $\mathcal{N}(P)$ is contained in the closed right half-plane
defined by $L$, and $L\cap \mathcal{N}(P)$\pedro{} is not empty, that
is a line meeting $\mathcal{N}(P)$ on its border.

{}
\begin{figure}[h!]
  \centering
  \begin{tikzpicture}[scale=0.7]
    \draw[step = 0.5,gray!30!white, thin] (0,0) grid (9,6);
    \foreach \i in {1,2,...,8}
    {\draw (\i,0) node[anchor=north] {$\i$};}
    \foreach \i in {1,2,...,6}
    {\draw (0,\i) node[anchor=east] {$\i$};}
    \draw(0,6) -- (0,0) -- (9,0);
    \draw[fill=black] (0,4) circle(3pt);
    \draw[fill=black](1,2) circle(3pt);
    \draw[fill=black] (3,1) circle(3pt);
    \draw[fill=black] (5,0) circle(3pt);
    \draw[fill=black] (3,6) circle(3pt);
    \draw[line width=1pt](9,6) -- (3,6) -- (0,4) -- (1,2) --
    (3,1) -- (5,0) -- (9,0);
    \draw[fill=gray!10!white, opacity=0.3]
    (9,6) -- (3,6) -- (0,4) -- (1,2) --
    (3,1) -- (5,0) -- (9,0);
    \draw[dashed] (-1,6) node[anchor=east]{$L(P;1/2)$} -- (2,0);
    \draw[dashed] (6,-.2) -- (-1,1.2) node[anchor=east] {$L(P;5)$};
    \draw[dashed] (4,6.5) -- (-1,4) node[anchor=east] {$L(P;-2)$};
    \draw[->](9,0) --(10,0) node[anchor=north]{$\alpha$};
    \draw[->](0,6) -- (0,7) node[anchor=east]{$|\rho|$};
  \end{tikzpicture}
  \caption{Cloud, Newton polygon and some supporting lines of $P$ in \eqref{eq:example}.}
  \label{fig:newton-polygon-1}
\end{figure}
Figure \ref{fig:newton-polygon-1} shows the points in the cloud and
the Newton polygon (bold lines) of the following polynomial (which
will be extensively studied in Section 5):
\begin{multline}\label{eq:example}
  P=-x^3\,{ Y_0}^4\,{ Y_5}^2
  +4\,{Y_1}^4
  -9\,{Y_0}^2\,{ Y_1}\,{ Y_2}
  +2\,{ Y_0}^3\,{ Y_2} \\
+q^{-4}{x{ Y_0}\,{ Y_2}}
-q^{-4}{x^3\,{ Y_2}}
-x^3\,{ Y_0}+x^5.
\end{multline}
{}Notice that the \pedro{}ordinate axis corresponds to $|\rho|$\pedro{}.

It will be convenient to speak about the \emph{co-slope} of a line
as the opposite of the inverse of its slope, the co-slope of a
vertical line being $0$. In order to deal with  the particular case
in which $P$ is a polynomial in the variables $Y_0,Y_1,\ldots,Y_n$ we
\pedro{}define:
\begin{equation*}
  \mu_{-1}(P) = 
  \left\{
    \begin{array}{l@{\ \ \ }l}
      -\infty & \mbox{if}\ P\ \mbox{is a polynomial in}\  Y_0,\ldots,Y_{n}\\
      0 & \mbox{otherwise}
    \end{array}
  \right.
\end{equation*}
Finally, from now on we assume $P\neq 0$ everywhere.
\pedro{}
\begin{lemma}\label{le:numero_finito_lados} 
  Let $P\in \mathcal{R}[[Y]]$. For any $\mu>\mu_{-1}(P)$
  there exists a unique supporting line of $\bar{\mathcal{C}}(P)$
  with co-slope $\mu$ and
  the Newton polygon $\mathcal{N}(P)$ has
  a finite number of sides with co-slope greater or equal than $\mu$. If
  $P$ is a polynomial then $\mathcal{N}(P)$ has a finite number of sides
  and vertices. If $P\in \mathcal{R}_\Gamma[[Y]]$ for some finitely
  generated semigroup $\Gamma\subseteq
  \mathbb{R}_{\geq0}$, then 
  the Newton Polygon $\mathcal{N}(P)$ has a finite number of sides with positive
  co-slope. 
\end{lemma}

The unique supporting line with co-slope $\mu$ will be denoted
henceforward $L(P;\mu)$. 
\begin{proof}
  If $P$ is a polynomial, let $h$ be its total degree {}in{}
  the variables
  $Y_0,\ldots,Y_n$. Otherwise we define $h$ as follows: 
  since $P\neq 0$ the set  $\mathcal{C}(P)$ is 
  nonempty; take a point  $q\in \mathcal{C}(P)$ and let $L$ be
  the line passing through $q$ with co-slope~$\mu$. Let $(0,h)$ be the
  intersection of $L$ with the $OY$-axis.  For each $\rho\in
  \mathbb{N}^{n+1}$, write $\alpha_\rho=\ord\,P_\rho(x)$.
  Only the finite number of  points $(\alpha_\rho,|\rho|)$ with
  $|\rho|\leq h$  {}and $P_{\rho}(x)\neq 0${} are relevant
  for the definition of the line $L(P;\mu)$ and for the construction of 
  sides 
  with co-slope greater or equal than $\mu$ of $\mathcal{N}(P)$. This
  proves the two first statements, 
  the last one is a
  consequence of the fact that for a given $\alpha>0$, the set
  $\Gamma\cap \{r<\alpha\}$ is finite.
\end{proof}

For $\mu>\mu_{-1}(P)$, define the following polynomial in the variable $C$:
\begin{displaymath}
  \Phi_{(P;\mu)}(C)=\sum_{(\alpha,|\rho|)\in L(P;\mu)}
  P_{\alpha,\rho}\,q^{\mu\,w(\rho)} \,C^{|\rho|}, 
\end{displaymath}
where $w(\rho)=\rho_1+2\rho_2+\cdots+n\rho_n$.
For a vertex $v$ of $\mathcal{N}(P)$, the indicial
polinomial is
$$
\Psi_{(P;v)}(T)=\sum_{(\alpha,|\rho|)= v}
P_{\alpha,\rho}\,T^{w(\rho)}.
$$ 

{}
For $P$ given in Equation \eqref{eq:example}, some examples of initial
and indicial polynomials are: for $v_{0}=(4,6)$,
$\Psi_{(P;v_0)}(T)=-3T^{10}$, and for $v_1=(0,4)$,
$\Psi_{(P;v_1)}(T)=T^2(T-2)(4T-1)$.
As regards the sides, the one joining $(3,6)$ with $(0,4)$,
  has co-slope $\gamma_1=-3/2$ and we have
  $\Phi_{(P;\gamma_1)}(C)=2q^{-3}C^{4}-9q^{-9/2}C^{4}+4q^{-6}C^{4}-q^{-15}C^6$,
  whereas  the one joining $(0,4)$ and $(1,2)$ has co-slope
  $\gamma_2=1/2$ and $\Phi_{(P;\gamma_2)}(C)=C^4(4q^2-9q^{3/2}+2q) + q^{-3}C^2$.
\pedro{}
\subsection{A rough idea of the method}\label{sub:rough-idea}
Newton's
algorithm is recursive in the following sense : assume $s(x)=cx^{\mu}+\overline{s}(x)$ is a solution of $P=P_{0}$ with $\ord_x\overline{s}(x)>\mu$. 
Then, on one side (see Lemma \ref{le:key_lemma}):
\begin{equation}\label{eq:nec-cond}
  \Phi_{(P;\mu)}(c) = 0,
\end{equation}
and on the other, $\overline{s}(x)$ is a solution of a new equation $P_1$ derived from
$P_0$ and $cx^{\mu}$ (see Corollary \ref{co:corolario-lema-clave}). The Newton Polygon is
a graphical tool to describe the necessary condition \eqref{eq:nec-cond} on $c$ and $\mu$:
if $\mu$ is the co-slope of a side of $\mathcal{N}(P)$, then \eqref{eq:nec-cond} is a
polynomial in $c$; if $\mu$ is a co-slope of a supporting line meeting $\mathcal{N}(P)$ at a vertex $v$, then \eqref{eq:nec-cond} becomes $\Psi_{(P;v)}(q^{\mu})=0$
(in this special cases, any coefficient is valid because $\Phi(P;\mu)\equiv 0$).

Iterating the above procedure, will allow us (see Proposition~\ref{pro:equivalencia_Procedure_PivotPoint}) to prove
that $S(x)\in\mathcal{R}$ is a solution of $P$ if and only if its support is countable (so
that we can write $S(x)=\sum_{i=0}^{\infty} a_ix^{\mu_i}$) and these conditions hold:
$\mu_i\rightarrow \infty$, and if we denote $S_j(x)=\sum_{i<j}a_{i}x^{\mu_i}$,
$P_j=P(S_j(x)+Y_0, \ldots, \sigma^n(S_j(x))+Y_n)$, then for all
$j\in \mathbb{Z}_{\geq 0}$:
\begin{equation}\label{eq:increase-order}
 \Phi_{(P_j;\mu_{j})}(a_{j})=0.
\end{equation}
The geometric meaning of
\eqref{eq:increase-order} is precisely (see Figure
\ref{fig:newton-polygon-2}), that the point
$L(P_j;\mu_j)\cap \{|\rho|=0\}$ is to the left of
$\mathcal{N}(P_{j+1})\cap \{|\rho|=0\}$, whereas
$\mu_i\rightarrow\infty$ implies that these points go to infinity.
Newton's idea consists of: instead of trying to compute a complete
solution straightaway, reduce the problem to computing each
$\mu_{j}$, $a_{j}$ iteratively, using the structure of $\mathcal{N}(P_j)$ and Equation
\eqref{eq:increase-order} each time (which is Procedure 1). The fact
that all solutions of $P$ can be found with this method is essentially
Proposition~\ref{pro:equivalencia_Procedure_PivotPoint}.

\pedro{}

\subsection{Composition} \label{sec:compositio-meaning} 
For $s_0,\ldots,s_n\in \mathcal{R}$,
the expression  
$P(s_0,\ldots,s_n)$ can be given a precise meaning under certain conditions.
We consider on $\mathcal{R}$ the topology
induced by the distance $d(f,g)=\exp(-\ord(f-g))$ which is
a complete topology: {}if $(f_n)$ is a Cauchy sequence, this means that given $M>0$, there is $N_{M}$ with $\ord(f_n-f_m)>M$ for any $n,m\geq N_{M}$; hence, for any $M>0$, the truncations of $f_n$ and $f_m$ up to order $M$ coincide, for $n,m\geq N_M$. Thus, there exists a single $f\in\mathcal{R}$ (defined inductively) such that $\ord(f_n-f)>M$ for $n\geq N_M$. This $f$ is the (unique limit) of the Cauchy sequence.{}

If $P$ is a polynomial, $P(s_0,\ldots,s_n)$ is  well-defined
because $\mathbb{C}((x^{\mathbb{R}}))$ is a ring. Otherwise,
we impose   $\ord(s_i)>0$, for
all $i$. Let $\mu=\min_{0\leq i\leq n}\{\ord(s_i)\}$. 
For $M\in \mathbb{N}$, consider the polynomial 
$P_{\leq M}=\sum_{|\rho|\leq M} P_\rho(x)\,Y^{\rho}$.  
The sequence
$P_{\leq M}(s_0,\ldots,s_n)$, $M\in 
\mathbb{N}$, is a Cauchy sequence
because the order of
$P_\rho(x)s_0^{\rho_0}\cdots s_n^{\rho_n}$ is greater than or equal to
$\mu\,|\rho|$. 
Its limit is precisely $P(s_0,\ldots,s_n)$. 
Notice that if  $P\in
\mathcal{R}_\Gamma[[Y]]$ and all $s_i\in
\mathcal{R}_\Gamma$, then 
$P(s_0,\ldots,s_n)\in \mathcal{R}_\Gamma$.

Given $s_0,\ldots,s_n$ as above, we define the series
\begin{equation}
  \label{eq:def-translacion-series}
  P(s_0+Y_0,\ldots,s_n+Y_n):=
  \sum_{\rho\in \mathbb{N}^{n+1}} \frac{1}{\rho!}
  \frac{\partial^{|\rho|}P}
  {\partial Y^{\rho}}
  (s_0,\ldots,s_n)\, \, Y^{\rho},
\end{equation}
where $\rho!=\rho_0!\cdots \rho_n!$ and
$\frac{\partial^{|\rho|}P}{\partial Y^{\rho}}= 
\frac{\partial^{|\rho|}P}
{\partial Y_0^{\rho_0}\,{}\partial{}Y_1^{\rho_1}\dots\partial Y_{n}^{\rho_n} }
$.
For generalized power series $\bar{s}_0,\ldots,\bar{s}_n$
with
positive order
it is 
straightforward to prove that the evaluation of the 
right hand side of~(\ref{eq:def-translacion-series})
at $\bar{s}_0,\ldots,\bar{s_n}$
is $P(s_0+\bar{s}_0,\ldots,s_n+\bar{s}_n)$.

If $y\in \mathbb{C}((x^{\mathbb{R}}))$ has $\ord(y)>\mu_{-1}(P)$, then
$P(y,\sigma(y),\ldots,\sigma^{n}(y))$ is  well defined
because $\ord(\sigma^{k}(y))=\ord(y)$. We also remark that if
$y\in \mathcal{R}_\Gamma$, then $\sigma^{k}(y)\in
\mathcal{R}_\Gamma$. The following notations will be used in the rest
of the paper:  
\begin{align}\nonumber
  P[y]&=P(y,\sigma(y),\ldots,\sigma^{n}(y)),
  \\\label{eq:def_translacion}
  P[y+Y]&=P(y+Y_0,\sigma(y)+Y_1,\ldots,\sigma^{n}(y)+Y_n).
\end{align}
We are also going to make use of the little-o notation: $o(x^{\mu})$
will mean a
generalized formal power series with order greater than $\mu$ or the
zero series if $\mu=\infty$.  
The following is essentially what motivates the Newton polygon construction:
\begin{lemma}\label{le:key_lemma}
  Let $y=c\,x^{\mu}+o(x^{\mu})\in\mathbb{C}((x^{\mathbb{R}}))$,
  and $\mu>\mu_{-1}(P)$.  
  Let $(\nu,0)$ be the intersection point of
  $L(P;\mu)$ with the $OX$-axis. 
  Then
  \begin{displaymath}
    P[y]=\Phi_{(P;\mu)}(c)\,x^{\nu}+o(x^{\nu}),
  \end{displaymath}
  In particular, if $y$ is a solution of the
  $q$-difference equation~(\ref{eq:q-difference-equation-intro}) then 
  \begin{displaymath}
    \Phi_{(P;\mu)}(c)=0.
  \end{displaymath}
\end{lemma}
\begin{proof}
  {}If $P$ is a polynomial, let $M$ be its total
  degree; otherwise, $\mu>\mu_{-1}(P)=0$  and we set $M$ as any integer $M$ such
that $M\mu> \nu$, say $M=\lfloor \nu/\mu \rfloor + 1$, where $\lfloor.\rfloor$ denotes the integral part.
{}
The truncation of $P[y]$ up to
  order $\nu$   is equal
  to that of $P_{\leq M}[y]$ and also
  $\Phi_{(P;\mu)}(C)=\Phi_{(P_{\leq M};\mu)}(C)$.

  Write $\alpha_\rho=\ord\,P_\rho$ for any multiindex $\rho$.
  {}
Recall that $L(P;\mu)=\{(\alpha,b)\mid   \alpha+\mu\,b=\nu\}$ is a
supporting line of $\mathcal{C}(P)$: this implies that for any $P_{\rho}\neq 0$, the point
$(\alpha_\rho,|\rho|)$ belongs to the closed right half-plane defined
by $L(P;\mu)$, from which follows that $\nu$ is the
  minimum of $\alpha_\rho+ \mu\,|\rho|$, for $\rho\in
  \mathbb{N}^{n+1}$.{}
  The following chain of equalities proves the result

  \begin{eqnarray*}
    &&P_{\leq M}[c\,x^{\mu}+o(x^{\mu})]=\\
    &&\sum_{|\rho|\leq M} 
    P_{\rho}(x)\,\, 
    (c\,x^{\mu}+o(x^{\mu}))^{\rho_0}
    (q^\mu c\,x^{\mu}+o(x^{\mu}))^{\rho_1}\cdots 
    (q^{n\mu} c\,x^{\mu}+o(x^{\mu}))^{\rho_n}=\\
    &&
    \sum_{|\rho|\leq M} 
    \left\{P_{\alpha_\rho,\rho} \, x^{\alpha_\rho}+ o(x^{\alpha_\rho})\right\}
    \left\{ c^{|\rho|} \, q^{\mu w(\rho)} \, x^{\mu|\rho|}+ o( x^{\mu|\rho|})\right\}
    =\\
    &&\sum_{|\rho|\leq M} 
    \left\{ P_{\alpha_\rho,\rho} \,
      c^{|\rho|} \, q^{\mu w(\rho)} \, x^{\alpha_\rho+\mu|\rho|}+ 
      o(x^{\alpha_\rho+\mu|\rho|})\right\}
    =\\
    &&\left\{
      \sum_{\alpha_\rho +\mu\,|\rho|=\nu} 
      P_{\alpha_\rho,\rho}\,\,
      c^{|\rho|}\,q^{ {}\mu{}\, w(\rho)} \right\}
    x^{\nu}+o(x^{\nu})=\Phi_{(P;\mu)}(c)+o(x^{\nu}).
  \end{eqnarray*}%
  {}where the last equality holds because, again
  $L(P;\mu)=\left\{ \alpha + \mu b = \nu \right\}$.{}
\end{proof}


  Let $y\in\mathbb{C}((x^{\mathbb{R}}))$ be a generalized power series
  and
  $S$ be its support. If $S$ is finite, denote by $\omega(y)$ the cardinal
  of $S$, otherwise $\omega(y)=\infty$. Consider the 
  sequence $\mu_i\in S$ defined inductively as follows: 
 $\mu_0$ is the minimum of $S$ and 
for $0\leq i<\omega(y)$, $\mu_{i+1}$ is the
  minimum of $S\setminus\{\mu_0,\mu_1,\ldots,\mu_{i}\}$. Let
  $c_i\in \mathbb{C}$ be the coefficient of $x^{\mu_i}$ in $y$.  
\begin{definition}We shall call 
   \emph{the first $\omega$ terms} of $y$  
  to the generalized power series $\sum_{0\leq
    i<\omega(y)}c_i\,x^{\mu_i}$.
  \end{definition}
Notice that if the support of $y$ is finite or has no accumulation points
then $y$ coincides with its first $\omega$ terms.

\begin{corollary}\label{co:corolario-lema-clave}
  Let $y$ be a solution of the $q$-difference
  equation~(\ref{eq:q-difference-equation-intro}) and let
  $\sum_{i}c_i\,x^{\mu_i}$ be the first $\omega$ terms of
  $y$. Let $P_i$ be the series defined as:
  \begin{displaymath}
    P_0:=P,\quad\text{and}\quad
    P_{i+1}:=P_i[c_i\,x^{\mu_i}+Y], 
    \quad 0\leq i<\omega(y).
  \end{displaymath}
  Then, for  all $0\leq i<\omega(y)$, one has
  \begin{displaymath}
    \Phi_{(P_{i};\mu_i)}(c_i)=0,\quad
    \text{and}\quad\mu_{i-1}<\mu_{i},
  \end{displaymath}
  where we denote $\mu_{-1}=\mu_{-1}(P)$.
\end{corollary}
\begin{proof}
  Let
  $\bar{y}_k=y-\sum_{i=0}^{k-1}c_i\,x^{\mu_i}$, then
  $P_k[ \bar{y}_k  ]=0$ and the first term of $\bar{y}_k$ is~$c_k\,x^{\mu_k}$.
\end{proof}
\pedro{} By way of example, consider, for $P$ given by
\eqref{eq:example}, the transformation with $\mu=2$ and $c=1$, which
gives $P_1=P[x^2+Y]$ having $33$ terms. The Newton polygon of $P_1$
(and its comparison to that of $P$) is given in Figure~\ref{fig:newton-polygon-2}. Observe how (this will be proved later as
Lemma \ref{le:NP_behaviour_under_translation}) the Newton Polygons
$\mathcal{N}(P)$ and $\mathcal{N}(P_1)$ coincide at and above the
vertex $v=(1,2)$, which is the topmost vertex of
$L(P;2)\cap \mathcal{N}(P)$. Underneath that vertex $v$, the point
$L(P;2)\cap \{|\rho|=0\} = (5,0)$ is to the left of
$\mathcal{N}(P_1)\cap \{|\rho|=0\}=(8,0)$.

At the same time, under $v$, the polygon $\mathcal{N}(P_1)$ has only
sides with co-slope greater than or equal to $2$ (in the example, just
one with co-slope $7/2$). As $\mu_1=2$, only co-slopes $\mu_j>2$ are chosen afterwards (see Section 5 and Figure \ref{picture1} for the complete example).

\begin{figure}[h!]
  \centering
  \begin{tikzpicture}[scale=0.6]
    \draw[step = 0.5,gray!30!white, thin] (0,0) grid (16,6);
    \foreach \i in {1,2,...,15}
    {\draw (\i,0) node[anchor=north] {$\i$};}
    \foreach \i in {1,2,...,6}
    {\draw (0,\i) node[anchor=east] {$\i$};}
    \draw(0,6) -- (0,0) -- (16,0);
    \draw[fill=black] (0,4) circle(3pt);
    \draw[fill=black](1,2) circle(3pt);
    \draw[fill=black] (8,0) circle(3pt);
    \draw[fill=black] (3,6) circle(3pt);
    \draw[line width=1pt](16,6) -- (3,6) -- (0,4) -- (1,2) -- (8,0) -- (16,0);
    \draw[fill=gray!10!white, opacity=0.3]
    (16,6) -- (3,6) -- (0,4) -- (1,2) -- (8,0) -- (16,0);
    \draw[line width=1pt,dashed](3.1,6) -- (0.1,4);
    \draw[line width=1pt,dashed] (0.05,4)-- (1.05,2);
    \draw[dashed,line width=1pt](1,2) -- (5,0);
    \draw[dashed,line width=1pt] (5,0.05) -- (8,.05);
    \draw[fill=none] (3,1) circle(3pt);
    \draw[fill=none] (5,0) circle(3pt);
    \foreach \x/\y in {2/3,4/2,6/1}
    {\draw[fill=gray,gray](\x,\y)circle(3pt);};
    \foreach \x/\y in {5/5,7/4,9/3,11/2,13/1,15/0}
    {\draw[fill=gray,gray](\x,\y)circle(3pt);};
    \draw[dashed] (5.5,-0.25) -- (-1,3) node[anchor=east]{$L(P;2)$};
  \end{tikzpicture}
  \caption{Cloud and Newton polygon $\mathcal{N}(P_1)$ of $P_1=P[x^2+Y]$ where $P$ is defined in \eqref{eq:example}. In dashed lines, $\mathcal{N}(P)$. Observe how both \emph{polygons} coincide at and above $(1,2)$, the topmost vertex of $L(P;2)\cap \mathcal{N}(P)$.}
  \label{fig:newton-polygon-2}
\end{figure}
\pedro{}

  Let $P\in \mathcal{R}_\Gamma[[Y]]$
  and let $\sum_{i=0}^{\infty}c_i\,x^{\mu_i}$ be a series with
  $\mu_{-1}(P)<\mu_i< \mu_{i+1}$, for all $0\leq i<\infty$
  (We do not impose that  $c_i\neq 0$, but the sequence 
  $(\mu_i)_{i\in \mathbb{N}}$ is strictly increasing).
  Consider the series $P_0:=P$ and $P_{i+1}:=P_i[c_ix^{\mu_i}+Y]$. 
\begin{definition}
We say that
  $\sum_{i=0}^{\infty}c_i\,x^{\mu_i}$ satisfies the 
  necessary initial conditions for $P$, in short $\NIC(P)$, if
  $\Phi_{(P_i;\mu_i)}(c_i)=0$, for all $i\geq 0$.  
\end{definition}
The above Corollary states that the first $\omega$ terms of a solution of
$P[y]=0$ satisfy $\NIC(P)$. In this section and the next one we shall
prove in Proposition~\ref{pro:NIC->solucion} the reciprocal statement 
for $P\in \mathcal{R}_{\Gamma}[[Y]]$: if
$\sum_{i=0}^{\infty}c_ix^{\mu_i}$ satisfies $\NIC(P)$, 
then  $\lim_{i\to\infty}\mu_i=\infty$ and
$\sum_{i=0}^{\infty}c_ix^{\mu_i}$ is an actual solution of the 
$q$-difference equation $P[y]=0$. This implies in particular that
solutions of $P[y]=0$ coincide with their first $\omega$ terms.

A method for computing all the series satisfying $\NIC(P)$ with 
$c_i\neq 0$, for all $i$, 
is the following one: 
\begin{proc}[Computation of a power series satisfying $\NIC(P)$]\mbox{ }\label{Procedure-1}\\
  Set $P_0:=P$ and $\mu_{-1}:=\mu_{-1}(P)$.
  \newline
  For $i=0,1,2, \ldots$ do either (a.1) or (a.2) and (b), where:
  \newline\indent
  (a.1). If $y=0$ is a solution of $P_i[y]=0$, then
  \texttt{return} $\sum_{k=0}^{i-1}c_kx^{\mu_k}$.
  \newline\indent
  (a.2). Choose $\mu_i>\mu_{i-1}$, and $0\neq c_i\in
  \mathbb{C}$ satisfying $\Phi_{(P_i,\mu_i)}(c_i)=0$.
  \newline\indent
  If neither (a.1) nor (a.2) can be performed
   then \texttt{return} \texttt{fail}.
  \newline\indent
  (b). Set $P_{i+1}(Y):=P_i[c_i\,x^{\mu_i}+Y]$.
\end{proc}

If  \texttt{fail} is returned at step $k$ of the above
Procedure, this means 
that there are no a solutions of $P[y]=0$ having
$\sum_{i=0}^{k-1}c_ix^{\mu_i}$ as its first $k$ terms. To prove this,
assume that $z$ is a solution having $\sum_{i=0}^{k-1}c_ix^{\mu_i}$
as its first $k$ terms. Either $z=\sum_{i=0}^{k-1}c_ix^{\mu_i}$, in
which case $y=0$ would be a solution of $P_k[y]=0$ and (a.1) would
have been performed, or
$z-\sum_{i=0}^{k-1}c_ix^{\mu_i}$ would have a first term of the form 
$c_kx^{\mu_k}$ so that
(a.2) could have been performed.

In order to carry out (a.2) in the above Procedure, one has to
deal with the following formula with quantifiers
\begin{equation}\label{eq:formulaQuantifiers}
  \exists \mu>\mu',\,
  \exists c\in \mathbb{C}, c\neq 0,
  \quad \Phi_{(P;\mu)}(c)=0.
\end{equation}
The Newton Polygon provides a way to eliminate the quantifiers. 
Fix $\mu'>\mu_{-1}(P)$; by
Lemma~\ref{le:numero_finito_lados}, $\mathcal{N}(P)$ has only a finite
number of sides $L_1,L_2,\ldots,L_t$ with co-slopes greater than $\mu'$. 
Let
$\gamma_1<\gamma_2<\ldots<\gamma_t$ be their respective co-slopes and
denote by $v_{i-1}$ and $v_i$ the
endpoints of $L_i$. Take $\mu>\mu'$. Either $\mu=\gamma_j$
for some $1\leq j\leq t$,
or  $\gamma_j<\mu<\gamma_{j+1}$
for some $0\leq j\leq t$ (writing
$\gamma_0=\mu'$ and $\gamma_{t+1}=\infty$). 
If $\mu=\gamma_j$, then $L(P;\mu)\cap
\mathcal{N}(P)=L_j$ and $\Phi_{(P;\mu)}(C)$ depends only on the
coefficients $P_{\alpha,\rho}$ of $P$ with $(\alpha,|\rho|)\in
L_j$. Otherwise,  $\gamma_j<\mu<\gamma_{j+1}$ for some $j$ and $L(P;\mu)\cap
\mathcal{N}(P)$ is just the vertex $v_j=(a,b)$, which implies that
{}
\begin{equation*}
\Phi_{(P;\mu)}(C)= C^{b}\cdot  \Psi_{(P;v_j)}(q^{\mu}).
\end{equation*}
From this equality follows that in order for $\Phi_{(P;\mu)}(c)$ to be $0$ for some $c\neq 0$, the co-slope $\mu$ must satisfy $\Psi_{P;v_j}(q^{\mu})=0$.
In other words:{} there exists $c\neq 0$ and $\mu$ with
$\gamma_j<\mu<\gamma_{j+1}$ such that $\Phi_{(P;\mu)}(c)=0$ if and 
only if there exists $\mu$, {}satisfying both $\gamma_j<\mu<\gamma_{j+1}$ and
$\Psi_{(P;v_j)}(q^{\mu})=0${}. This proves that
equation~(\ref{eq:formulaQuantifiers}) 
is equivalent to the quantifier-free formula obtained by the
disjunction of the following formul{\ae}:
\begin{align}\label{eq:nec_cond_lados}
  \Phi_{(P;\gamma_j)}(c)&=0, \quad &1\leq j\leq t,\\
  \label{eq:nec_cond_vertices}
  \Psi_{(P;v_j)}(T)&=0,\,\mu=\log T/\log
  q,\,\gamma_j<\mu<\gamma_{j+1},
  \quad &0\leq j\leq t.
\end{align}

\subsection{The pivot point}
\pedro{} We prove in this subsection that
if $Q_{0}$ is the topmost vertex of $L(P;\mu_0)\cap \mathcal{N}(P)$ and
$P_1=P_{1}[y]$ is the first substitution, then $Q_{0}$ is also the topmost
vertex of $L(P_1;\mu_0)\cap \mathcal{N}(P_1)$, as exemplified in Figure \ref{fig:newton-polygon-2}. This allows one to give
a \emph{descent} argument which guarantees that, from some index $j_0$ on,
the point $Q_j$ (the topmost in
$L(P_{j};\mu_{j})\cap \mathcal{N}(P_{j})$) is equal to $Q_{j_0}$ for
$j\geq j_0$ (i.e. $Q_{j}$ remains the same for $j\geq j_0$). This fixed vertex will be called the \emph{pivot point}, as
for $j>j_0$ on, each supporting line $L(P_j;\mu_{j})$ ``hinges'' around it
 when the substitution
$P_{j}\rightarrow P_{j+1}$ is carried out. The existence of this pivot
point (and what we call \emph{relative pivot points} in subsection
\ref{sub:relative}) guarantees the finiteness properties of Theorems 1
and 2.

In fact, we prove later that if $s(x)$ is a solution of $P$, then
either the pivot point has ordinate equal to $1$ or we
can derive a new equation from $P$ which also has $s(x)$ as a
solution and whose pivot point with respect to $s(x)$ has ordinate
equal to $1$. This simplifies our arguments considerably because when this happens, (\ref{eq:increase-order}) is linear in $a_j$.  \pedro{}

For $P\in \mathcal{R}_\Gamma[[Y]]$ and
$\mu>\mu_{-1}(P)$, we shall denote by $Q(P;\mu)$ the point with highest
ordinate in $L(P;\mu)\cap \mathcal{N}(P)$. For
$\bar{P}=P[c\,x^{\mu}+Y]$ (as in
equation~(\ref{eq:def_translacion})), the following Lemma describes the
Newton Polygon of $\bar{P}$:

\begin{lemma}\label{le:NP_behaviour_under_translation}
  Let $h$ be the ordinate of $Q(P;\mu)$ and consider the
  half-planes
  $h^{+}=\{(a,b)\in \mathbb{R}^2\mid b\geq h\}$,
  $h^{-}=\{(a,b)\in \mathbb{R}^2\mid b\leq h\}$. If 
  $L(P;\mu)^{+}$ is the closed right
  half plane defined by $L(P;\mu)$ and $(\nu,0)$ is the intersection of
  $L(P;\mu)$ with the $OX$-axis, then
  \begin{enumerate}
  \item $\mathcal{N}(\bar{P})\cap h^{+}=\mathcal{N}(P)\cap h^{+}$,
    in particular $Q(P;\mu)\in \mathcal{N}(\bar{P})$. Moreover, for
    any $\alpha$ and $\rho$ with
    $(\alpha,|\rho|)=Q(P;\mu)$, the coefficients 
    $P_{\alpha,\rho}$ and $\bar{P}_{\alpha,\rho}$ are equal.
  \item $\mathcal{N}(\bar{P})\cap h^{-}\subseteq L(P;\mu)^{+}\cap h^{-}$,
  \item The point $(\nu,0)\in \mathcal{N}(\bar{P})$ if and only if
    $\Phi_{(P;\mu)}(c)\neq 0$. 
  \end{enumerate}
\end{lemma}
\begin{proof}
  Write $M_\rho(Y)=P_\rho(x)Y^{\rho}$ 
  and $\alpha_\rho=\ord P_\rho(x)$. It is
  straightforward to show that 
  $M_\rho[cx^{\mu}+Y]=
  M_\rho(Y)+V(Y)$ for some $V(Y)$,
  whose cloud 
  of points is contained in the set 
  $A_\rho=\{(a,b)\mid b<|\rho|\} \cap L(M_\rho;\mu)^{+}$. 
  This proves  part (2). If $Q=(\alpha,\rho)$ belongs to $\mathcal{N}(P)\cap
  h^{+}$, then there are no points $Q'=(\alpha',\rho')\in \mathcal{N}(P)$,
  except $Q$ itself, such that
  $Q\in A_{\rho'}$. This proves part~(1). Part (3) is a consequence of
  Lemma~\ref{le:key_lemma}. 
\end{proof}

\begin{corollary}\label{co:DefinicionPivotPoint}
  Let $\bar{\mu}>\mu$. Then either
  $Q(P;\mu)=Q(\bar{P},\bar{\mu})$ or the ordinate of
  $Q(\bar{P},\bar{\mu})$ is less than the ordinate of
  $Q(P;\mu)$. If  $\Phi_{(P;\mu)}(c)\neq 0$, then the
  ordinate of $Q(\bar{P};\bar{\mu})$ is zero. 
\end{corollary}
\begin{proof}
  The previous Lemma implies that $Q(P;\mu)$ is a vertex of
  $\mathcal{N}(\bar{P})$ and $L(P;\mu)=L(\bar{P};\mu)$. 
  Hence $Q(P;\mu)=Q(\bar{P};\mu)$. Since $\bar{\mu}>\mu$, 
  $Q(\bar{P};\bar{\mu})$ is a vertex with ordinate less than or equal to
  the ordinate of $Q(\bar{P};\mu)=Q(P;\mu)$. For the second
  part, assume that $\Phi_{(P;\mu)}(c)\neq 0$. By the same Lemma, the point
  $(\nu,0)\in \mathcal{N}(\bar{P})$, so that the segment whose
  endpoints are $(\nu,0)$ and $Q(\bar{P};\mu)$ is the
  only side of $\mathcal{N}(\bar{P})$ with co-slope greater than or
  equal to $\mu$, from which follows that $Q(\bar{P};\bar{\mu})=(\nu,0)$.
\end{proof}

  Let $P\in \mathcal{R}_\Gamma[[Y]]$
  and take a series $\psi(x)=\sum_{i=0}^{\infty}c_i\,x^{\mu_i}$  with
  $\mu_{-1}(P)<\mu_i< \mu_{i+1}$ for all $0\leq i<\infty$.
  (Notice that we do not impose that  $c_i\neq 0$, but the sequence 
  $(\mu_i)_{i\in \mathbb{N}}$ must be strictly increasing).
  Writing $P_0:=P$ and $P_{i+1}:=P_i[c_ix^{\mu_i}+Y]$,
  let $Q_i=Q(P_i;\mu_i)$. By the previous Corollary, the ordinate of {}$Q_{i+1}$
  is less than or equal to the ordinate of $Q_{i}${}.
  Since these are
  natural numbers, there exists $N$ such that for $i\geq
  N$, the ordinate of $Q_i$ is equal to the ordinate of $Q_N$ (it
  stabilizes).
  By the same Corollary, we know that actually $Q_N=Q_i$, for all
  $i\geq N$. This leads to the following 
  \begin{definition}\label{def:pivot_point}
  The \emph{pivot point of $P$ with respect to
  $\psi(x)$} is the point $Q$ at which the sequence $Q_i$ stabilizes and is
denoted by $Q(P;\psi(x))$. We say that it is reached at step $N$ if
$Q_N=Q(P;\psi(x))$.
  \end{definition}
  Let $Q_N=(\alpha,h)$ be the pivot point just defined. From part (1) of
  Lemma~\ref{le:NP_behaviour_under_translation} follows that
  $(P_N)_{\alpha,\rho}=(P_i)_{\alpha,\rho}$ for all $i\geq N$,
  and for all $\rho$ with
  $|\rho|=h$. In particular, the indicial polynomials
  $\Psi_{(P_i;Q_N)}(T)$ are the same for all $i\geq N$.
  We shall say that \emph{the monomial $Y^{\rho}$ (resp. the variable $Y_j$)
  appears effectively in 
  the pivot point} if $(P_N)_{\alpha,\rho}\neq 0$ (resp. for some
  $\rho$ with $\rho_j>0$).
\begin{proposition}\label{pro:equivalencia_Procedure_PivotPoint}
  Let $P$ and $\psi(x)=\sum_{i=0}^{\infty}c_i\,x^{\mu_i}$ be as above.
  The following statements are equivalent: 
  \begin{enumerate}
  \item The ordinate of the pivot point of $P$ with respect to
    $\sum_{i=0}^{\infty}c_i\,x^{\mu_i}$ is greater than or equal to $1$.
  \item The series $\sum_{i=0}^{\infty}c_i\,x^{\mu_i}$ satisfies
    $\NIC(P)$.
  \end{enumerate}
In case $\lim \mu_i=\infty$, these statements are equivalent to
\begin{enumerate}
\item[(3)]  The series $\psi(x)$ is a solution of $P[y]=0$.
\end{enumerate}
\end{proposition}
\begin{proof}
  Assume statement (1). The ordinate of $Q_{i+1}$ is non-zero and 
  by the above Corollary, $\Phi_{(P_i;\mu_i)}(c_i)=0$, which
  proves~(2). Assume now that statement (1) is false, so that the ordinate
  of the pivot point is zero. This means that there exists some $N$ such
  that $Q_N$ has ordinate zero. By definition of $Q_N$ we have that
  $L(P_N;\mu_N)\cap \mathcal{N}(P_N)$ is just the point
  $Q_N=(\alpha,0)$. Then $\Phi_{(P_N;\mu_N)}(C)$ is a non-zero constant
  (namely the coefficient of $x^{\alpha}$ in $P_N$), therefore it has no
  roots, in contradiction with  $\Phi_{(P_N;\mu_N)}(c_N)=0$. This
  proves the equivalence between (1) and (2).
  By Corollary~\ref{co:corolario-lema-clave}, (3) implies (2).

  Assume  (1) holds and that $\lim \mu_i=\infty$.  
  Write
  $\psi_k(x)=\sum_{i=0}^{k-1}c_ix^{\mu_i}$ and notice that
  $P_i=P[\psi_i(x)+Y]$, 
  in particular, $P[\psi_i(x)]=P_i[0]=(P_i)_{\underline{0}}$.
  Let $Q=(\alpha,h)$ be the pivot point of $P$ with respect to
  $\psi(x)$. 
  Since $L(P_i;\mu_i)$ contains
  the point $Q$, $\ord
  (P_i)_{\underline{0}} >\alpha+h\mu_i$ and since $h\geq 1$, the
  sequence $ \ord
  P[\psi_i(x)]$ tends to infinity and we are done.
\end{proof}
\begin{corollary}\label{co:solution->ordinatePivotPoint>1}
  Let $\sum_{i=0}^{\infty}c_i\,x^{\mu_i}$ be the
  first $\omega$-terms of a solution of $P[y]=0$.
  Then the pivot point of $P$ with
  respect to $\sum_{i=0}^{\infty}c_i\,x^{\mu_i}$ has ordinate greater
  than or equal to $1$.  
\end{corollary}

\subsection{Relative pivot points}\label{sub:relative}
The above construction of the pivot point can be made relative to any
of the 
variables $Y_j$, $0\leq j\leq n$, and more generally, relative to any
monomial $Y^{r}$, with $r=(r_0,r_1,\cdots,r_n)\in
\mathbb{N}^{n+1}$, as follows:

Fix $r\in \mathbb{N}^{n+1}$.
The cloud of points of $P$
relative to $Y^{r}$ is defined as the set 
$\mathcal{C}_r(P)=
\{(\alpha,|\rho|)\mid \exists
\rho,\,\text{with }P_{\alpha,\rho}\neq 0,\text{ and } r\preceq \rho\}$,
where $r\preceq \rho$ means that $r_i\leq \rho_i$, for all $0\leq
i\leq n$. 
It is obvious that
$\mathcal{C}_r(P)\subseteq \mathcal{C}(P)$. 

Assume that $\mathcal{C}_r(P)$ is not the empty set, then
we may 
define  the line $L_r(P;\mu)$ as the 
leftmost line with co-slope
$\mu$ having nonempty intersection with $\mathcal{C}_r(P)$. 
The point $Q_r(P;\mu)$ will be the one with greatest ordinate in
$L_r(P;\mu)\cap \mathcal{C}_r(P)$.


If $H$ denotes $H=\frac{\partial^{|r|} P}{\partial Y^{r}}$, the cloud
$\mathcal{C}_r(P)$ is not the empty set if and only if
$H$ is not the zero series. In this case,
consider the translation map
$\tau(a,b)=(a,b-|r|)$. It is straightforward to prove that 
$\mathcal{C}_r(P)=\tau^{-1}(\mathcal{C}(H))$. 
Hence,
$L_r(P;\mu)=\tau^{-1}(L(H;\mu))$, and
$Q_r(P;\mu)=\tau^{-1}(Q(H;\mu))$.

Let $\psi(x)=\sum_{i=0}^{\infty}c_i\,x^{\mu_i}$, with $\mu_0>\mu_{-1}(P)$.
Denote $H_0=H$ and
$H_{i+1}=H_i[c_ix^{\mu_i}+Y]$. By the chain
rule,%
\begin{equation}\label{eq:chain_rule}
  \frac{\partial^{\pedro{|r|}{}} P_{i}}{\partial Y^{r}}=H_{i},\quad i\geq 0.
\end{equation}
The sequence of points
$Q_r(P_i;\mu_i)=\tau^{-1}(Q(H_i;\mu_i))$ for $i\geq 0$ stabilizes at
some point denoted $Q_r(P;\psi(x))$ and which we
call \emph{the
pivot point of $P$ with respect to $\psi(x)$ relative to $Y^{r}$}.
Therefore
\begin{equation}
  \label{eq:relacion_pivot_point_pivot_point_derivada}
  Q(H;\psi(x))=\tau(Q_r(P;\psi(x)))
\end{equation}
\begin{remark}\label{re:bien_definido_PivotPointRelativo}
  Since $H\neq 0$, then $H_i\neq 0$, for $i\geq 0$ so that
  $\mathcal{C}_r(P_i)$ is not empty, for $i\geq 0$. This proves that
  $Q_r(P;\mu_i)$ and $Q_r(P;\psi(x))$ are well-defined provided the
  monomial $Y^{r}$ appears effectively in $P$.
\end{remark}

From now on, we shall denote $e_j$ the vector $(0,\dots, 0, 1, 0,
\dots, 0)$ where the $1$ appears at position $j+1$, for $j=0,\dots,
n$. Thus, $e_j=(\delta_{ij})_{0\leq i\leq n}\in \mathbb{N}^{n+1}$ 
where $\delta_{ij}$ is the Kronecker symbol. 

\begin{proposition}\label{pro:reducion_altura_1}
  Let $Q=(a,h)$ be the pivot point of $P$ with respect
  to $\psi(x)$. Assume that the monomial
  $Y^{r'}$ appears effectively in $Q$. Let $r\in \mathbb{N}^{n+1}$, with
  $r\preceq r'$, and  $H=\frac{\partial^{|r|} P}{\partial Y^{r}}$.
  Then the pivot point of $H$ with respect to $\psi(x)$ is
  $(a,h-|r|)$. 
  In particular, 
if $r=r'-e_i$, for some $i$ such that $r'_i\geq 1$, then the
  ordinate of the pivot point $Q(H;\psi(x))$ is $1$.
  However,  for $r=r'$, one has $Q(H;\psi(x))=(a,0)$ and
  therefore $\psi(x)$ is not a solution of $H[y]=0$. 
\end{proposition}
\begin{proof} 
  Assume the pivot point $Q$ is reached at step $N$, thus $Q\in
  \mathcal{C}_{r'}(P_i)\subseteq \mathcal{C}_{r}(P_i)$ for
  all $i\geq N$. From $\mathcal{C}_r(P_i)\subseteq \mathcal{C}(P_i)$
  and the fact that $Q=Q(P_i;\mu_i)$ for all $i>N$, one infers
  $Q=Q_r(P_i;\mu_i)=Q_{r'}(P_i;\mu_i)$ for all $i>N$. This means that $Q$ is
  the pivot point of 
  $P$ with respect to $\psi(x)$ relative to $Y^{r}$ and also relative
  to $Y^{r'}$.  
  As we have seen before, $\tau_r(Q)=(a,h-|r|)$ is
  the pivot point of $H$ with respect to $\psi(x)$. 
 The third
  statement is a consequence of
  Proposition~\ref{pro:equivalencia_Procedure_PivotPoint}.
\end{proof}

\begin{corollary}
  Let $\psi(x)=\sum_{i=0}^\infty c_i\,x^{\mu_i}$ be a solution of $P[y]=0$ with
  $\lim\mu_i=\infty$. If the pivot point
  $(P;\psi(x))$ has ordinate  greater than $1$, then there exists a
  non trivial derivative $H=\frac{\partial^{|r|} P}{\partial Y^{r}}$
  of  $P$,  such that $\psi(x)$ is a solution of $H[y]=0$.
\end{corollary}
\begin{proof}
 Let $Y^{r'}$ be a monomial that appears effectively in the pivot
 point $Q=Q(P;\psi(x))$. Since $Q$ has ordinate greater that $1$, $r'$
 can be chosen with $|r'|\geq 2$. Let $r$ be such
 that $r\preceq r'$ and $1\leq |r|<|r'|$. 
 By the Proposition, the pivot point of $H$ with
 respect to $\psi(x)$ has ordinate greater than or equal to $1$.
 By Proposition~\ref{pro:equivalencia_Procedure_PivotPoint}, $\psi(x)$ is a
 solution of $H[y]=0$.
\end{proof}

\begin{lemma}\label{le:relativePivotPoints}
  Let $Q(P;\psi(x))=(a,b)$ and $Q_r(P;\psi(x))=(a',b')$ be
  respectively the general pivot point of $\psi(x)$ and the pivot
  point of $\psi(x)$ relative to $Y^{r}$.
  If  the sequence $\mu_i$ of exponents of $\psi(x)$ tends to infinity,
  then {}the following two statements hold:{}
  \begin{itemize}
    {}
  \item \pedro{}The ordinate of $Q_r(P;\psi(x))$ is at least
    $b$:\pedro{} $b'\geq b$, and
  \item If $b'=b$ \pedro{}(both points are at the same height)\pedro{}, then $a' \geq a$.
    {}
  \end{itemize}
\end{lemma}
\begin{proof}
  Assume that both pivot points have been reached at step $N$.
  For any $i\geq N$, the point $(a',b')$ belongs to the closed right
  half plane  {}$L(P_{i};\mu_i)^{+}${} because  $\mathcal{C}_r(P_i)\subseteq
  \mathcal{C}(P_i)$. Since $(a,b)\in {}L(P_i;\mu_i){}$ for all $i\geq N$,
  and $\lim \mu_i=\infty$, the intersection of all the half planes
  {}$L(P_i;\mu_i)^{+}${} for 
  $i\geq N$, is the region {}$R${} formed by the  points in {}$L(P_N;\mu_N)^{+}${} 
  with ordinate greater than or equal to~$b$. {} The result follows because 
$(a',b')\in R$, and $(a,b)$ is the most left point of $R$ with
ordinate equal to~$b$. {}
\end{proof}

\section{Finiteness  properties}
Throughout this section,  we assume that $\Gamma$ is a finitely
generated semigroup of $\mathbb{R}_{\geq 0}$ and that $P$ is a nonzero
element of $\mathcal{R}_\Gamma[[Y]]$. We also assume that $q\neq 1$:  the
case $q=1$ is
reduced to the case $n=0$  considering 
$P(Y_0,Y_0,\ldots,Y_0)$.
This section is devoted to proving the following results:
\begin{theorem}\label{th:solutions_are_grid_based}
  If $y\in\mathbb{C}((x^{\mathbb{R}}))$ is a solution of
  equation~(\ref{eq:q-difference-equation-intro}), then 
  it is a grid-based formal power series.
\end{theorem}

\begin{proposition}\label{pro:NIC->solucion}
  If $\psi(x)=\sum_{i=0}^{\infty} c_ix^{\mu_i}$ satisfies $\NIC(P)$,
  then $\psi(x)$ is a solution of $P[y]=0$.
\end{proposition}


\begin{definition}\label{de:finite-determination}
   Let $y\in \mathbb{C}((x^{\mathbb{R}}))$ and $P\in
  \mathcal{R}_{\Gamma}[[Y]]$. We say that $y$ is \emph{finitely determined
  by $P$} if there exist positive integers $k$ and $h$, such that if
$y_k$ denotes the first $k$ terms of $y$ then $y$ is the only element
$z\in \mathbb{C}((x^{\mathbb{R}}))$  satisfying the following property:
{}
``$z_k=y_k$ and  $Q[y]=0$ if and only if: for any 
$Q=\frac{\partial^{|r|} P}{\partial Y^{r}}$, with $|r|\leq h$, one has  $Q[z]=0$.''
 {}
\end{definition}

\begin{theorem}\label{th:solutions_are_finitely_determined}
  If $\abs{q}\neq 1$, then any solution $y$ of
  equation~(\ref{eq:q-difference-equation-intro}) is finitely
  determined by $P$.
\end{theorem}

The hypothesis $|q|\neq 1$ is
necessary: let $P=Y_0-Y_1$ and $q=\sqrt{-1}$. Any series $\sum_{i=0}^\infty
c_{4i}\,x^{4i}$ (for arbitrary constants $c_{4i}$) is a  solution of
$P[y]=0$.
Since $\partial P/\partial Y_0 [y]=0$ and  $\partial
P/\partial Y_1[y]=0$ have no solutions, and  higher order
derivatives of $P$ are zero, none of these solutions is  
finitely determined by~$P$. {}If $|q|=1$, $q^{\alpha}=1$ for
$\alpha >0$ irrational, and $q\neq 1$, then $\sum_{i=0}^{\infty}
a_ix^{i\alpha}$ is a also a solution of $P[y]=0$ for any sequence
$a_i$, and it is not finitely determined either. {}


\begin{remark}\label{re:Gamma_numerable}
  Let $\Gamma$ be a finitely generated semigroup of $\mathbb{R}_{\geq
    0}$. For any real number $k$, the set $\Gamma\cap \{r\mid r\leq
  k\}$ is finite. Hence $\Gamma$ is a well-ordered set with no
  accumulation points and its  elements can be enumerated
  in increasing order: $\Gamma=\{\gamma_i\}_{i\geq 0}$, with
  $\gamma_{i}<\gamma_{i+1}$ and $\lim \gamma_i=\infty$. Let
  $\psi(x)=\sum_{i=0}^{\infty} c_ix^{\mu_i}$ be the 
  first $\omega$ terms of an element
  $y\in\mathcal{R}$. 
  If $\supp \psi(x)$ is contained in $\Gamma$ then
  either it is finite or $\lim \mu_i=\infty$. In both
  cases, $y=\psi(x)$. In particular, any element of
  $\mathcal{R}$ whose support is contained in 
  $\Gamma$ coincides with its first $\omega$ terms.
\end{remark}



\subsection{Quasi solved form}\label{subsec:quasi_solved_form}
\pedro{} Once we know that the pivot point $Q$ corresponding to the
solution $s(x)$ can be assumed to have ordinate $1$, we perform a
transformation on $P$ sending $Q$ to $(0,1)$. Any equation whose pivot
point with respect to a solution is at $(0,1)$ is very easy to study,
as the successive Newton polygons only change below that point. This,
together with the ease of computing their solutions is what makes this
property relevant and deserving its own name, \emph{quasi-solved
  form}.

A special case of \emph{quasi-solved form}, called \emph{solved form},
guarantees also that $P$ has a unique solution $s(x)$ with $s(0)=0$.
If $P$ has integer exponents and is in solved form, then it
has a single solution $s(x)$ with $s(0)=0$ and its exponents are
integer (i.e. $s(x)$ is a formal power series). As a side note,
solutions to equations in solved form are studied in depth in our book
\cite{Barbe-Cano-Fortuny-McCormick-arxiv} (their asymptotic
properties, radius of convergence, etc.). In fact, many
power series arising from combinatorial problems are in (or are easily
turned into) solved form. We refer to
\cite{Barbe-Cano-Fortuny-McCormick-arxiv} for the details.  \pedro{}

We say that the
equation
\begin{equation}
  \label{eq:solved_form:1}
  P[y]=0,\quad \ord(y)>0,
\end{equation}
is in \emph{quasi-solved form} if the point $(0,1)$ is a vertex of
$\mathcal{N}(P)$ and $(0,0)\not\in \mathcal{C}(P)$. If this is the
case, let $\Psi(T)$ be
the indicial polynomial of $P$ at 
$(0,1)$,  $\Sigma=\{\mu\in \mathbb{R}\mid \Psi(q^{\mu})=0\}$ and 
$\Sigma^{+}=\Sigma\cap \mathbb{R}_{>0}$. 
We say that equation~(\ref{eq:solved_form:1}) is in
\emph{solved form} if $\Sigma^{+}$ is the empty set.
{}
One can verify (but it is irrelevant to our purposes) that an equation in solved form has a unique grid-based power series solution.
{}

{}For the sake of comparison, a \emph{linear} equation $Q=\sum
a_i(x)\sigma^{j}$ is in quasi-solved form if
$a_j(0)\neq 0$ for some  $j\geq 1$.{}

{}
The proof of Theorems 1 and 2 is structured as follows. A technical lemma on finitely generated semigroups allows us to introduce a change of variable
$z=x^{\gamma}\,y$ which will allow us to reduce the problem to quasi-solved form. Then we show (Lemma \ref{le:quasi_solved_form_case}) that the solution is grid-based in this case. We also obtain in this case (Corollary \ref{co:solved_form_recursive_formula}) a recursive formula for the coefficients of the solution. Finally, the proofs of Theorems 1 and 2 follow.
{}

\begin{remark}\label{re:set_Sigma}
  The polynomial $\Psi(T)$ can be written
  $\Psi(T)=P_{0,e_0}+P_{0,e_1}\,T+\cdots + P_{0,e_n}\,T^{n}
  \in
  \mathbb{C}[T]$. Its degree $m$ is 
  the largest index such that the variable $Y_m$ appears effectively in
  the point $(0,1)$. If the equation is in quasi-solved form,
  $\Psi(T)$ is a nonzero polynomial
  because  $(0,1)\in \mathcal{C}(P)$.
  If $|q|\neq 1$,  then
  $\Sigma$ is finite. If case $|q|=1$ (and $q\neq 1$), then
  $\Sigma$ is the finite union of the sets
  $\Sigma_r=\frac{\arg(r)}{\arg(q)}+\frac{2\pi}{\arg(q)}\mathbb{Z}$, for
  those complex roots $r$ of  $\psi(T)$ with modulus $1$.
  Recall that we
  have fixed  a determination of the logarithm
  to compute $q^{\mu}$, hence $\arg(q)$ is also fixed. The following Lemma
  implies that $\Sigma_r\cap\mathbb{R}_{\geq 0}$ is
  {}contained in{}
  a finitely generated semigroup. Therefore $\Sigma^{+}$ generates a
  finitely generated semigroup of $\mathbb{R}_{\geq 0}$.
\end{remark}

\begin{lemma}\label{le:shiftFinitelyGeneratedSemigroups}
  Let $\gamma\in \mathbb{R}$ and $\gamma_1,\gamma_2,\ldots,\gamma_s$
  positive 
  real numbers. Then the
  semigroup  $\Gamma$ of $\mathbb{R}_{\geq 0}$ generated by the set
  $
  A=(\gamma+\gamma_1\mathbb{N}+\cdots+\gamma_s\mathbb{N})\cap
  \mathbb{R}_{\geq 0}
  $
  is finitely generated. 
\end{lemma}
\begin{proof}
  Let $\Lambda$ be the
  set of $(n_1,\ldots,n_s)\in \mathbb{N}^{s}$ such that $\gamma+\sum
  n_i\gamma_i>0$. By Dickson's lemma, the number of minimal elements
  in $\Lambda$ with respect the product order are finite. Hence
  $\Gamma$ is generated by $\gamma_1,\ldots,\gamma_s$ and  the family
  $\gamma+\sum n_i\gamma_i>0$ for all minimal element
  $(n_1,\ldots,n_s)$ of $\Lambda$. 
\end{proof}

{}We now introduce a change of variables which will allow us to simplify the exponents of the $x$ variable in an equation{}.
Let $P\in
\mathcal{R}_{\Gamma}[[Y]]$ and $\gamma>\mu_{-1}(P)$. 
Define $P[x^{\gamma}Y]$ as the series 
\begin{equation}\label{eq:changeOfVariableMshif}
  \sum_{\rho}q^{\gamma\,\omega(\rho)}\,x^{\gamma|\rho|}\,P_\rho(x)\, 
  Y_0^{\rho_0}Y_1^{\rho_1}\cdots  Y_n^{\rho_n}
  \in \mathbb{C}((x^{\mathbb{R}}))^{g}\,\,[[Y]].
\end{equation}
If $(\nu,0)$ is
the intersection point of $L(P;\gamma)$ with the $OX$-axis, then all
the coefficients of the series $P[x^{\gamma}Y]$ have order greater
than or
equal to $\nu$. Define
$\mshift{P}{\gamma}=x^{-\nu}P[x^{\gamma}Y]$. 
The coefficients of $\mshift{P}{\gamma}$ are in
$\mathcal{R}_{\Gamma^{*}}$, where $\Gamma^{*}$ is the semigroup of
$\mathbb{R}_{\geq 0}$ generated by
$(-\nu+\Gamma+\gamma\mathbb{N})\cap \mathbb{R}_{\geq 0}$. By 
Lemma~\ref{le:shiftFinitelyGeneratedSemigroups}, $\Gamma^{*}$ is a
finitely generated semigroup of $\mathbb{R}_{\geq 0}$.

The transformation $P\mapsto \mshift{P}{\gamma}$ corresponds to the 
change of variable $z=x^{\gamma}y$ in the following sense: for a
series $y$, with $\ord y>\gamma+\mu_{-1}(P)$, one has
$\mshift{P}{\gamma}[x^{-\gamma}y]= x^{-\nu}P[y]$, in particular,  $P[y]=0$ if and only if
$\mshift{P}{\gamma}[x^{-\gamma}y]=0$.

Let $\bar{\tau}(a,b)$ be
the plane affine map $\bar{\tau}(a,b)=(a-\nu+\gamma b,b)$, which satisfies
$\bar{\tau}(\mathcal{C}_j(P))=\mathcal{C}_j(\mshift{P}{\gamma})$ for
$0\leq j\leq n$. In particular,
$\bar{\tau}(\mathcal{N}(P))=\mathcal{N}(\mshift{P}{\gamma})$, and
$\bar{\tau}$ maps vertices to vertices and sides of co-slope $\mu\geq
\gamma$ to sides 
of co-slope $\mu-\gamma$. Moreover, 
$\bar{\tau}(L(P;\mu))=L(\mshift{P}{\gamma};\mu-\gamma)$,
in particular $\bar{\tau}(L(P;\gamma))=L(\mshift{P}{\gamma};0)$
is the vertical axis. Therefore, 
$Q(P;\mu)$ and $Q(\mshift{P}{\gamma};\mu-\gamma)$ have the same
ordinate. 
Let $\sum_{i=0}^{\infty}c_ix^{\mu_i}$ and $P_i$ be as in the definition
of pivot point (Definition~\ref{def:pivot_point}). Assume $\gamma<\mu_0$ and set
 $H=\mshift{P}{\gamma}$,
$H_0=H$ and $H_{i+1}=H_i[c_ix^{\mu_i-\gamma}+Y]$. 
It is straightforward to prove that $\mshift{P_i}{\gamma}=H_i$, so that
$\bar{\tau}(Q(P_i;\mu_i))=Q(H_i;\mu_i-\gamma)$ and, in particular, they
have the same ordinate. Then the image by $\bar{\tau}$ of  
the pivot point of $P$ with respect to $\sum
_{i=0}^{\infty}c_ix^{\mu_i}$  is the pivot
point of $\mshift{P}{\gamma}$ with respect to
$\sum_{i=0}^{\infty}c_ix^{\mu_i-\gamma}$ and the same holds for
relative pivot points.
By Proposition~\ref{pro:equivalencia_Procedure_PivotPoint}, this implies that 
$\sum_{i=0}^{\infty}c_ix^{\mu_i}$ satisfies $\NIC(P)$ if and only if
$\sum_{i=0}^{\infty}c_ix^{\mu_i-\gamma}$ satisfies $\NIC(\mshift{P}{\gamma})$.

Finally,  if $v\in \mathcal{C}(P)$ then
$\bar{\tau}(v)\in 
\mathcal{C}(\mshift{P}{\gamma})$ and 
$\Psi_{(\mshift{P}{\gamma};\bar{\tau}(v))}(T)=\Psi_{(P;v)}(q^{\gamma}T)$.

\begin{lemma}\label{le:reduction_to_solved_form}
  Assume that  $\psi(x)=\sum_{i=0}^{\infty}c_i\,x^{\mu_i}$ satisfies
  $\NIC(P)$. 
  Then there exist a finitely generated semigroup
  $\Gamma^{*}$, a 
  series 
  $P^{*}\in \mathcal{R}_{\Gamma^{*}}[[\Y]]$, an index $N$ and a
  rational number $\gamma$ with $\mu_{N-1}\leq \gamma<\mu_{N}$, such that
  the equation
  \begin{equation}
    \label{eq:reduction-solved-form}
    P^{*}[z]=0,\quad \ord z>0
  \end{equation}
  is in quasi solved form and 
  $\psi^*(x)=\sum_{i=N}^{\infty}c_i\,x^{\mu_i-\gamma}$ satisfies $\NIC(P^{*})$. 
\end{lemma}
\begin{proof}
  We may assume that the ordinate of the pivot point of
  $\psi(x)$ with respect to $P$ is
  $1$. Otherwise, by
  Proposition~\ref{pro:reducion_altura_1}, 
  we may replace $P$ by any of its derivatives
  $\frac{\partial^{|r|}P}{\partial Y^{r}}$, 
  where  the monomial $Y_j\,Y^{r}$
  appears effectively in the pivot point, for some $j$. 
  We remark that the
  coefficients of any derivative of $P$ also belong to
  $\mathcal{R}_{\Gamma}$. 
  Let $Q=(\alpha,1)$ be the pivot point of $P$ with respect to $\psi(x)$
  and use the notation of Definition~\ref{def:pivot_point}: 
  $P_0=P$, $P_{i+1}=P_i[c_ix^{\mu_i}+Y]$ and so on. In particular,
  let 
  the pivot point be  reached at step~$N'-1$ for some $N'$. 
  Consider any integer $N\geq N'$.
  Denote $\Gamma_0=\Gamma$ and
  $\Gamma_{i+1}=\Gamma_{i}+\mu_i\,\mathbb{N}$. Notice that the
  coefficients of 
  $P_i$ belong to $\mathcal{R}_{\Gamma_i}$.

  Let $\gamma$ be a rational number such that $\mu_{N-1}\leq\gamma<\mu_{N}$
  and set $P^{*}=\mshift{P_N}{\gamma}\in \mathcal{R}_{\Gamma_N^{*}}[[Y]]$. 
  Since the pivot point $Q$ has been
  reached at step $N-1$, 
  $Q\in  L(P_{N-1};\mu_{N-1})\cap L(P_N;\mu_N)$. By
  Lemma~\ref{le:NP_behaviour_under_translation},  $Q\in
  L(P_N;\mu_{N-1})$. Hence  $Q\in L(P_N;\mu_{N-1})\cap L(P_N;\mu_N)$;
  since 
  $\mu_{N-1}<\gamma<\mu_{N}$, we conclude that $Q(P_N;\gamma)=Q=(\alpha,1)$. 
  {}So, as the change of variables ~(\ref{eq:changeOfVariableMshif}) sends a point $(a,b)$ to $\tau(a,b)=(a-\nu+\gamma b,b)$ for the corresponding $\nu$, we get $\tau(\alpha,1)=(0,1)$ and{}
  the point $(0,1)$, {}so that{} is in $\mathcal{C}(P^{*})$, 
  the equation
  $P^{*}[y]=0$ is in quasi solved form and the pivot point of $P^{*}$
  with respect $\psi^*(x)$  is $(0,1)$. By
  {}Proposition{}~\ref{pro:equivalencia_Procedure_PivotPoint},
   $\psi^*(x)$ satisfies $\NIC(P^{*})$.
\end{proof} 

\begin{lemma}\label{le:quasi_solved_form_case}
  Assume
  equation~(\ref{eq:reduction-solved-form}) is in quasi-solved
  form and  let $\xi(x)=\sum_{i=0}^{\infty}c_ix^{\mu_i}$, with $\mu_0>0$,
  be a series 
  satisfying $\NIC(P^{*})$.  
  Then the support of $\xi(x)$ is contained in
  the finitely generated semigroup 
  $\Gamma'=\Gamma^{*}+\Sigma^{+}\,\mathbb{N}$. 
  In particular, either the support of
  $\xi(x)$ is finite or $\lim \mu_i=\infty$ and in both cases $\xi(x)$ is a
  solution of equation~(\ref{eq:reduction-solved-form}).
\end{lemma}
\begin{proof}
  Let $P_0=P^{*}$ and $P_{i+1}=P_i[c_ix^{\mu_i}+Y]$ for $i\geq 0$.
  We first prove that $Q(P_i;\mu_i)=(0,1)$ for all $i\geq 0$. We do this
  showing, by induction on $i$, that 
  $\mathcal{N}(P_i)$ is contained into the first 
  quadrant of the plane and that the point $(0,1)\in
  \mathcal{C}(P_i)$. 
  This holds for $P_0$ because of 
  the hypotheses on $P^{*}$. Assume that the statement 
  holds for $P_i$. Since $\mu_i>0$, the line 
  $L(P_i;\mu_i)$ either contains the point $(0,1)$, and then
  $Q(P_i;\mu_i)=(0,1)$,  or  
  $L(P_i;\mu_i)$ meets $\mathcal{N}(P_i)$ at a single point with zero
  ordinate which is $Q(P_i;\mu_i)$. If the
  latter happens, from Corollary~\ref{co:DefinicionPivotPoint},
  we infer that the pivot point of $P^{*}$ with respect to $\xi(x)$
  has zero ordinate, in 
  contradiction with the fact that $\xi(x)$ satisfies $\NIC(P^{*})$. 
  Hence $Q(P_i;\mu_i)=(0,1)$. By
  Lemma~\ref{le:NP_behaviour_under_translation}, $(0,1)$ is  a vertex of 
  $\mathcal{N}(P_{i+1})$ and since $P_{i+1}\in \mathcal{R}[[Y]]$, its Newton polygon
  is contained in the first quadrant. This proves the induction step
  and  that  $Q(P_i;\mu_i)=(0,1)$,~$i\geq 0$. 

  The fact that $Q(P_i;\mu_i)=(0,1)$ implies that the polynomial
  $\Phi_{(P_i;\mu_i)}(C)$ is equal to 
  $\Psi(q^{\mu_i})C+\Coeff(P_i;x^{\mu_i}\,Y^{\underline{0}})$,
  where $\Psi(T)$ is the indicial polynomial of $P$ at~$(0,1)$
  and $\Coeff(P_i;x^{\mu_i}\,Y^{\underline{0}})$ is the coefficient
  of $x^{\mu_i}\,Y_0^{0}Y_1^{0}\cdots Y_n^{0}$
  in $P_i$. 
  Since $\Phi_{(P_i;\mu_i)}(c_i)=0$ because $\xi(x)$
  satisfies $\NIC(P^{*})$, the following equations hold:
  \begin{equation}
    \label{eq:ecuacion_recursiva:1}
    \Psi(q^{\mu_i})\,c_i+\Coeff(P_i;x^{\mu_i}\,Y^{\underline{0}})=0,\quad
    i\geq 0.
  \end{equation}
  Let us prove, by induction, that $P_i\in
  \mathcal{R}_{\Gamma'}[[Y]]$, for all $i\geq 0$,
  and that the support of $\xi(x)$ is contained in $\Gamma'$. By
  hypothesis, $P_0\in \mathcal{R}_{\Gamma'}[[Y]]$. Assume that $P_i\in
  \mathcal{R}_{\Gamma'}[[Y]]$. If $c_i=0$, then $P_{i+1}=P_i\in
  \mathcal{R}_{\Gamma'}[[Y]]$ and $\mu_i\not\in\supp(\xi(x))$. If, on
  the contrary,
  $c_i\neq 0$, we can prove by contradiction that $\mu_i\in
  \Gamma'$: assume that $\mu_i\not \in \Gamma'$, in particular
  $\mu_i\not\in \Sigma^{+}$, hence $\Psi(q^{\mu_i})\neq 0$. From
  equation~(\ref{eq:ecuacion_recursiva:1}),
  $\Coeff(P_i;x^{\mu_i}\,Y^{\underline{0}})\neq 0$, {}and
  therefore{}
  $\mu_i\in \supp((P_i)_{\underline{0}})\subseteq \Gamma'$. So
  $P_{i+1}=P_i[c_ix^{\mu_i}+Y]$ belongs to
  $\mathcal{R}_{\Gamma'}[[Y]]$ which proves the induction step. 

  The set $\supp \xi(x)$  has no accumulation points in
  $\mathbb{R}$ because 
  $\Gamma'$ is a finitely generated semigroup of
  $\mathbb{R}_{\geq 0}$ and $\supp \xi(x)\subseteq \Gamma'$ and we are done.
\end{proof}

\begin{corollary}\label{co:solved_form_recursive_formula}
  Let $y$ be a solution of equation
  (\ref{eq:reduction-solved-form}) which is in quasi solved form. Let
  $\Gamma'=\{\gamma_i\}_{i=0}^{\infty}$, with $\gamma_i<\gamma_{i+1}$
  for all $i$. Then $y=\sum_{i=1}^{\infty} d_i\,x^{\gamma_i}$ with
  $d_i$ satisfying the following recurrent formula:
  \begin{equation}
    \label{eq:recurrent_formula}
    \Psi(q^{\gamma_i})\,d_i=-\Coeff(P^{*}[d_1x^{\gamma_1}+\cdots+d_{i-1}x^{\gamma_{i-1}}];
    x^{\gamma_i}), \quad i\geq 1.
  \end{equation}
  If $\Sigma^{+}$ is finite and $z$ is another solution
  of equation (\ref{eq:reduction-solved-form}) with $\ord
  (y-z)$ greater than any element of $\Sigma^{+}$, then $y=z$.
\end{corollary}
\begin{proof}
  Let $\xi(x)$ be the first $\omega$ terms of $y$. Then $\supp
  \xi(x)\subseteq \Gamma'$, and by Remark~\ref{re:Gamma_numerable},
  $y=\xi(x)\in \mathcal{R}_{\Gamma'}$. Hence we may write 
  $y=\sum _{i=1}^{\infty} d_i x^{\gamma_i}$ because $\gamma_0=0$ and
  $\ord y>0$. Since $\xi(x)$ satisfies $\NIC(P^{*})$, the same
  reasoning as in Lemma~\ref{le:quasi_solved_form_case} up to
  equation~(\ref{eq:ecuacion_recursiva:1}) holds. The coefficient
  $\Coeff(P_i;x^{\gamma_i}Y^{\underline{0}})$ is equal to the
  coefficient of $x^{\gamma_i}$ in
  $P^{*}[d_1x^{\gamma_1}+\cdots+d_{i-1}x^{\gamma_{i-1}}]$, which gives
  equation~(\ref{eq:recurrent_formula}). 
  To prove the last statement, write $z=\sum _{i=1}^{\infty} d'_i
  x^{\gamma_i}$. 
  If $\gamma_{i}$ is
  greater than any element of $\Sigma^{+}$, then
  $\Psi(q^{\gamma_i})\neq 0$, and $d_i$ is completely determined
  by $d_1,\ldots,d_{i-1}$, so that~$y=z$. 
\end{proof}

\begin{proof}[Proof of Proposition~\ref{pro:NIC->solucion}]
  Applying Lemma~\ref{le:reduction_to_solved_form} to
  $\psi(x)$ we obtain equation~(\ref{eq:reduction-solved-form}), and
  applying 
  Lemma~\ref{le:quasi_solved_form_case} to  $\xi(x)=
  \sum_{i=N}^{\infty}c_ix^{\mu_i-\gamma}$ we conclude that $\mu_i-\gamma\in
  \Gamma'$, for $i\geq N$. 
  Since $\gamma\geq \mu_0$, the set $(\gamma-\mu_0)+\Gamma'$ is
  included in $\mathbb{R}_{\geq 0}$. 
  Let $\Gamma''$ be the semigroup generated by
  $(\gamma-\mu_0)+\Gamma'$. By
  Lemma~\ref{le:shiftFinitelyGeneratedSemigroups}, $\Gamma''$ is finitely
  generated. Let $\Gamma'''$ be the finitely generated semigroup
  $\Gamma''+\sum_{i=0}^{N-1} (\mu_i-\mu_0)\,\mathbb{N}$. 
  The set $\supp
  \psi(x)$ is contained in 
  $\mu_0+\Gamma'''$, so that $\lim \mu_i=\infty$.
  By Proposition~\ref{pro:equivalencia_Procedure_PivotPoint},
  $\psi(x)$ is a solution of $P[y]=0$.
\end{proof}

\begin{proof}[Proof of
  Theorem~\ref{th:solutions_are_grid_based}]
  Let $\psi(x)=\sum_{i=0}^{\infty}c_ix^{\mu_i}$ be the first $\omega$ terms
  of $y$. By 
  Corollary~\ref{co:corolario-lema-clave}, $\psi(x)$ satisfies
  $\NIC(P)$.  As in the proof of Proposition~\ref{pro:NIC->solucion},
  there exists a finitely generated semigroup $\Gamma$ such that 
  $\supp \psi(x)$ is contained in $\mu_0+\Gamma$.
  By Remark~\ref{re:Gamma_numerable}, $y=\psi(x)$, so that
  $y$ is  grid-based.
\end{proof}

\begin{proof}[Proof of Theorem~\ref{th:solutions_are_finitely_determined}]
  Let $y$  be a solution of
  equation~(\ref{eq:q-difference-equation-intro}). By
  Theorem~\ref{th:solutions_are_grid_based}, $y$ coincides with its
  first $\omega$ terms. Write
  $y=\sum_{i=0}^{\infty} c_i x^{\mu_i}$ and let $Q=(\alpha,h)$ be the
  pivot point 
  of $P$ with respect to $y$. 
  Apply
  Lemmas~\ref{le:reduction_to_solved_form} and
  \ref{le:quasi_solved_form_case} to $y$:  
  let $N$ and $\gamma$ be  as in
  Lemma~\ref{le:reduction_to_solved_form}; 
  we may assume that the
  pivot point $Q$ is reached at step $N-1$.  
  Since $|q|\neq 1$, 
  $\Sigma^{+}$ is finite  by Remark \ref{re:set_Sigma}.  Since $\lim
  \mu_i=\infty$, there is $k>N$ such 
  that $\mu_k-\gamma$ is greater than any element of~$\Sigma^{+}$.  

  Consider $z\in \mathbb{C}((x^{\mathbb{R}}))$ with  
  the same first $k$ terms as $y$ and satisfying  
that for any 
  $H=\frac{\partial^{|r|} P}{\partial Y^{r}}$, with $|r|\leq h$,
  $H[y]=0$ if and only if $H[z]=0$.  We have to show that $y=z$. 

  Since $P[y]=0$, then  $P[z]=0$, 
  and $z$ coincides with its first $\omega$
  terms. Write
  $z=\sum_{i=0}^{\infty} d_i x^{\delta_i}$. By hypothesis, $c_i=d_i$ and
  $\mu_i=\delta_i$ for $0\leq i<k$. Denote $P'_0=P$,
  $P'_{i+1}=P_i[d_ix^{\delta_i}+Y]$ and $P_0=P$
  and $P_{i+1}=P_i[c_ix^{\mu_i}+Y]$, for $i\geq 0$. 
  Obviously, $P_i=P'_i$, for $0\leq
  i \leq k$. In particular $Q=Q(P_N;\mu_N)=Q(P'_N;\delta_N)$.

  If the pivot point of $P$ with respect to $z$
  is also $Q$,
  then apply Lemmas~\ref{le:reduction_to_solved_form}
  and~\ref{le:quasi_solved_form_case} to $z$ in the same way as to
  $y$: choose the same derivative $\frac{\partial^{h-1} P}{\partial
    Y^{r}}$, the same $N$ and the same $\gamma$ to obtain the same
  $P^{*}$.
  This can be done because  $P_i=P'_i$, for $0\leq i \leq k$.
  This implies that
  $\xi(x)=\sum_{i=N}^{\infty} {c_i}x^{\mu_i-\gamma}$ and
  $\bar{\xi}(x)=\sum_{i=N}^{\infty} {d_i}x^{\delta_i-\gamma}$ both
  satisfy $\NIC(P^{*})$. By 
  Corollary~\ref{co:solved_form_recursive_formula},
  $\xi(x)=\bar{\xi(x)}$, which implies $y=z$.

  Let us show by contradiction that the pivot point $Q'$ of $P$ with
  respect to $z$ must be $Q$. Assume $Q'\neq Q$.
  {}
  The point $Q'$ is the
stabilization point of the sequence $Q(P'_i;\delta_i)$. On the other hand, $Q$ 
  belongs to this sequence because $Q=Q(P_N;\mu_N)=Q(P'_N;\delta_N)$.
  Corollary~\ref{co:DefinicionPivotPoint} implies that
  either $Q=Q'$ or otherwise
  their ordinates satisfies $h>h'$. Hence $h\geq h'$.{}
  Let $Y^{r}$ be a monomial that
  appears effectively in the pivot point of $P$ with respect to $z$,
  so that $|r|=h'$. Let $H=\frac{\partial
    ^{h'} P}{\partial Y^{r}}$. By 
  Proposition~\ref{pro:reducion_altura_1}, $H[z]\neq 0$; in particular
  $H\neq 0$. We claim that
  $H[y]=0$. By Remark~\ref{re:bien_definido_PivotPointRelativo}, the
  pivot point $Q_r$ of $P$ with respect to $y$ relative to $Y^{r}$ is
  well-defined. Since $\lim \mu_i=\infty$, by
  Lemma~\ref{le:relativePivotPoints}, the ordinate of $Q_r$ is
  $h''\geq 
  h$.  The pivot
  point of $H$ with respect to $y$ has ordinate 
  $h''-h'\geq h-h'\geq 1$. By Proposition~\ref{pro:equivalencia_Procedure_PivotPoint},
  $y$ satisfies $\NIC(P)$ and so $H[y]=0$, which proves our claim and
  finishes 
  the proof the the Theorem.
\end{proof}

\subsection{Bounding the rational rank in order and degree
  $1$}
\pedro{} In general, it is nice to know a priori how complex a
solution of an equation can be. Following Seidenberg \cite{Seidenberg-1968}, one can deduce that if $s(x)$ is a solution of a differential equation
$P=A(x,y)+B(x,y)y^{\prime}$ of order and degree $1$ with
$A(x,y),B(x,y)\in \mathbb{C}[[x,y]]$, then its support is included
in the $\mathbb{Q}$-vector space $\mathbb{Q} + \alpha \mathbb{Q}$, for
some $\alpha\in \mathbb{R}$. Morally speaking, one can only have a
single irrational exponent (and its $\mathbb{Q}$-span) in $s(x)$. We
pose here the same question (in all its generality, allowing $P$ to
have exponents in a finitely generated semigroup) and reach the
equivalent conclusion: the dimension of the vector space generated by
the support of a solution $s(x)$ is at most $1$ plus the dimension of
the vector space generated by the support of $P$.  \pedro{}

Recall that the rational rank of a semigroup $S\subseteq \mathbb{R}$
is the dimension of $\vsg{S}$, the  $\mathbb{Q}$-vectorial subspace of
$\mathbb{R}$ generated by $S$. It is denoted $\RatRk(S)$.

In what follows $\Gamma$ denotes a finitely generated semigroup of
$\mathcal{R}_{\geq 0}$, as above.

\begin{theorem}\label{th:bound-rati-rank} Assume $|q| \neq 1$.
  Let $P=A(Y_0)+B(Y_0)Y_1$ be a nonzero series, where $A,B\in
  \mathcal{R}_{\Gamma}[[Y_0]]$. Let $y$ be a solution of $P[y]=0$,
  with $\ord y > \mu_{-1}(P)$. Then $\RatRk(\supp y)\leq\RatRk(\Gamma)+1$.
\end{theorem}
{}The inequality can be strict, as witness the equation $P=Y_1-q^{\pi}Y_0$, which has as solution $y(x)=x^{\pi}$.
{}
\begin{proof}
  By the previous results, $y$ coincides with its first $\omega$ terms
  $\psi(x)=\sum_{i=0}^\infty c_i \,x^{\mu_i}$. 
 Taking a rational $\gamma<\mu_0$ and replacing $P$ by
 $\mshift{P}{\gamma}$ we may assume that $\mu_0>0$ and
 that $\mshift{P}{\gamma}\in \mathcal{R}_{\Gamma^{*}}$ and
 $\RatRk(\Gamma^{*})=\RatRk(\Gamma)$, for another finitely generated
 semigroup $\Gamma^{*}$. 
 
  Define $P_0=P$,
   $P_{i+1}=P_i[c_ix^{\mu_i}+Y]$, $\Gamma_0=\Gamma^{*}$ and $\Gamma_{i+1}=\Gamma_i+\mu_i\mathbb{N}$. 
The coefficients of $P_i$ belong
  to $\mathcal{R}_{\Gamma_i}$.  
 Notice that one has
  $\dim \vsg{\Gamma_{i+1}}\leq \dim \vsg{\Gamma_i}+1$ and the
  inequality holds only if $\mu_i\not \in\vsg{\Gamma_{i}}$.

  For each {}index{} $i$, the line $L(P_i;\mu_i)$ corresponds
  either to a vertex or to a side of $\mathcal{N}(P_i)$. If it
  corresponds to a side, then there are two different points $(\alpha,a)$
  and $(\beta,b)$ in
  $\mathcal{C}(P_i)$ lying on $L(P_i;\mu_i)$. 
  This implies that $\alpha,\beta\in
  \Gamma_i${}, therefore{}
  $\mu_i=(\beta-\alpha)/(a-b)\in \vsg{\Gamma_i}${}, and
  $\vsg{\Gamma_i}=\vsg{\Gamma_{i+1}}$.{}
Hence it is enough to prove that if for an index $i$,
   $\mu_i$ corresponds to a vertex of $\mathcal{N}(P_i)$, then 
   for all $j>i$, $\mu_j$ corresponds to a side of $\mathcal{N}(P_j)$.
    this holds, iterating the above argument, we get
   $\dim\vsg{\Gamma_0}=\dim\vsg{\Gamma_i}$ and 
     $ \dim\vsg{\Gamma_{i+1}}=\dim\vsg{\Gamma_j}$, for $j>i$, which completes the proof because we have:
     \begin{displaymath}
            \dim \vsg{\cup_{j=0}^{\infty} \Gamma_j}=\dim\vsg{\Gamma_{i+1}}\leq
            1+\dim\vsg{\Gamma_i}=
            1+\dim\vsg{\Gamma_0}.
          \end{displaymath}
          
          {}
          We now prove the statement above. Assume that for the index $i$,
     $\mu_{i}$ corresponds to a vertex  $v=(a,h)$
  of $\mathcal{N}(P_i)$.
By Corollary~\ref{co:corolario-lema-clave},
we have that $\Phi_{(P_i;\mu_i)}(c_i)=0$. 
Applying next Lemma~\ref{re:rango_uno} to $P_i$ we obtain that
$v'=(\nu-h,1)$ is a vertex of $\mathcal{N}(P_{i+1})$
and that $\Psi_{(P_{i+1};v')}(q^{\mu})\neq 0$, for $\mu> \mu_i$.
By Lemma~\ref{le:NP_behaviour_under_translation},
$\mathcal{N}(P_{i+1})$ is contained in the closed right half-plane
defined by $L(P_i;\mu_i)$. Since $v'\in L(P_i;\mu_i)$ and
$\mu_{i+1}>\mu_i$, then the point $Q(P_{i+1};\mu_{i+1})$ is either
$v'$ or a point with zero ordinate. The last possibility would be in
contradiction with the fact that $\psi(x)$ is a solution of $P=0$
and Proposition~\ref{pro:equivalencia_Procedure_PivotPoint}, hence
$Q(P_{i+1};\mu_{i+1})=v'$ and is ordinate is $1$. But this implies that
$Q(P_{i+1};\mu_{i+1})$ is the
pivot point of $P$ with respect to $\psi(x)$, from which follows that for $j>i$,
$Q(P_j;\mu_j)=v'$ and $\Psi_{(P_j;v')}=\Psi_{(P_{i+1};v')}$ (see
Definition~\ref{def:pivot_point} and the subsequent paragraph).

Let us prove, finally, that for any $j>i$,
$\mu_j$  corresponds to a side of
$\mathcal{N}(P_j)$. Were this not the case, for some $j>i$,
$\mu_j$ would correspond either to a vertex $v'$
or to a different vertex with zero ordinate, and both possibilities are absurd:
\begin{itemize}
\item If $\mu_j$ corresponds to the vertex $v'$ then
$$\Phi_{(P_j;\mu_j)}(c_j)=\Psi_{(P_j;v')}(q^{\mu_j})\,c_j
=\Psi_{(P_{i+1};v')}(q^{\mu})\,c_j\neq 0$$
which contradicts
Proposition~\ref{pro:equivalencia_Procedure_PivotPoint}.
\item If $\mu_j$ corresponds to a vertex with zero ordinate, then
$\Phi_{(P_j;\mu_j)}$ is a non-zero constant polynomial and has no roots, but by hypothesis $a_j$ is indeed a root.
\end{itemize}
Hence
$\mu_j$  corresponds to a side of
$\mathcal{N}(P_j)$, and we are done.
{}
\end{proof}

\begin{lemma}\label{re:rango_uno}
   Let $P=A(Y_0)+B(Y_0)\,Y_1$ where $A(Y_0),B(Y_0)\in
  \mathcal{R}_\Gamma[[Y_0]]$. Let $\mu>\mu_{-1}(P)$ such that
  $L(P;\mu)\cap \mathcal{N}(P)$ is a vertex $v=(a,h)$ of
  $\mathcal{N}(P)$ and let $c$ be a nonzero constant such that
  $\Phi_{(P;\mu)}(c)=0$. Let $\bar{P}=P[c\,x^{\mu}+Y]$.
  Then the point $v'=(\nu-h,1)$, where $\nu=a+\mu\,h$, is a vertex of
  $\mathcal{N}(\bar{P})$. Moreover, for $\mu'>\mu$ we have that
  $\Psi_{(\bar{P},v')}(q^{\mu'})\neq 0$.
\end{lemma}
\begin{proof}
  Since 
  $L(P;\mu)\cap \mathcal{N}(P)=\{v\}$, we have that
  $\Phi_{(P;\mu)}(C)=\Psi_{(P;v)}(q^{\mu})\,C^{h}$.
  By hypothesis,  
  $\Phi_{(P;\mu)}(c)=0$ and $c\neq 0$, hence
  $\Psi_{(P;v)}(q^{\mu})=0$.
  Let us denote $M=A\,x^{a}\,Y_0^{h}+B\,x^{a}\,Y_0^{h-1}\,Y_1$, $A,B\in
  \mathbb{C}$,  to be the sum of the terms of $P$
  corresponding to the vertex $v$; in particular, either $A$ or $B$ is
  different from zero.
  Then $\Psi_{(P;v)}(T)=A+B\,T$ and  $q^{\mu}$ is the unique
  root of $\Psi_{(P;v)}(T)$.

  {}
  Lemma~\ref{le:NP_behaviour_under_translation}
  describes $\mathcal{N}(\bar{P})$: The point $(\nu,0)\not\in
  \mathcal{C}(\bar{P})$ because $\Phi_{(P,\mu)}(c)=0$ and
  $\mathcal{N}(\bar{P})$ is contained in the closed right-half plane
  defined by $L(P;\mu)$. Let us prove that the point $v'=(\nu-h,1)\in
  \mathcal{C}(\bar{P})$ which would prove that $v'$ is a vertex of
  $\mathcal{N}(\bar{P})$ because it belongs to $L(P;\mu)$. For that,
  we compute the monomials of $\bar{P}$ corresponding to the point
  $v'$. Again by Lemma~\ref{le:NP_behaviour_under_translation}, these
  are the monomials of $\bar{M}=M[c\,x^{\mu}+Y]$ corresponding
  to~$v'$.  By direct computation these monomials are
  $c^{h-1}\,x^{\nu-h}\,(A\,Y_0+B\,Y_1)$.
  Since $c\neq 0$ and either $A\neq 0$ or $B\neq 0$ then
 $v'\in \mathcal{C}(\bar{P})$.
 Moreover $\Phi_{(\bar{P},v')}(T)=c^{h-1} (A+B\,T)$ and then $q^{\mu}$
is the only root of $\Phi_{(\bar{P},v')}(T)$. Since $|q|\neq 1$, if
$\mu'>\mu$, then $q^{\mu'}\neq q^{\mu}$ and then
$\Phi_{(\bar{P},v')}(q^{\mu'})\neq 0$.
{}
\end{proof}

\section{$q$-Gevrey order}\label{sec:qGevreyOrder}%
Throughout this section we assume that $|q|>1$
In this case, we prove
some properties about the growth of the
coefficients of a  formal power series  solution of a
$q$-difference equation.
{}
Note that the case $|q|<1$ follows from this one
by considering the 
equation $P(q^{-n}x,\sigma^{-n}(y),\ldots,y)=0$
which is equivalent to equation~\eqref{eq:q-difference-equation-intro}
because the operator $\sigma^{-1}(y(x))=y(q^{-1}\,x)$. 
{}

\begin{definition}\label{def:q-gevrey-order}
A formal
power series $\sum_{i=0}^{\infty}c_i\,x^{i}$ is said to be of $q$-Gevrey
order $s\geq 0$ if the series
$\sum_{i=0}^{\infty}c_i\,|q|^{-\frac{1}{2}s\,i^2}x^{i}$ has a positive
radius of convergence.  

We will say that a series
$P=\sum_{\alpha,\rho}P_{\alpha,\rho}\,x^{\alpha}\,Y^{\rho}\in \mathbb{C}[[x,\Y]]$ 
is of $q$-Gevrey  order $s\geq 0$ if the series
$$
\sum_{(\alpha,\rho)\in \mathbb{N}\times \mathbb{N}^{n+1}} P_{\alpha,\rho}\, 
|q|^{-\frac{1}{2} s(\alpha+|\rho|)^{2}}
\,
x^{\alpha}\,Y^{\rho}$$ 
has a positive radius of convergence at the
origin of $\mathbb{C}^{n+2}$.
\end{definition}

We remark that $q$-Gervey of order $0$ means convergence.
This section is devoted to proving the following result (the
number $s(P;y(x))$ in the statement is introduced in
Definition~\ref{de:q-gevrey_bound} 
and can be computed from the relative pivot points of $P$ with respect
to $y(x)$).

\begin{theorem}\label{th:Gevrey-order}
  Let $P\in \mathbb{C}[[x,\Y]]$ be a non-zero formal power series of
  $q$-Gevrey order $t\geq 0$ and
  $y(x)\in \mathbb{C}[[x]]$ a solution of
  $P[y]=0$. Then $y(x)$ is of
  $q$-Gevrey order $t+s(P;y)$ (see the following definition).
\end{theorem}

\begin{definition}\label{de:q-gevrey_bound}
  Let $Q=(a,h)$ be the pivot point of $P$ with respect to $y(x)$. 
  The number $s(P;y)$ is defined as follows:

  \textbf{Case $h=1$}. Let $Q_j=(a_j,h_j)$ be the pivot
  point of $P$ with respect to $y(x)$ relative to the variable
  $Y_j$ (for $0\leq j\leq n$). Since $Q$ has ordinate $1$, $Q=Q_j$ for some $j$. 
  Let $r=\max\{j\mid Q_j=Q\}$. There are three cases:
  \begin{enumerate}
  \item[(RS-R)] If $r=n$, then $s(P;y(x))=0$.
  \item[(RS-N)] If $r<n$ and $h_j>1$ for all $r<j\leq n$, then
    $s(P;y(x))$ can be taken as any positive number.
    {}In this case, Theorem \ref{th:Gevrey-order} says that $y(x)$
    is of $q$-Gevrey order $t+\varepsilon$ for any $\varepsilon>0$.{}
  \item[(IS)] If $r<n$ and $h_j=1$ for some $r<j\leq n$, then
  $s(P;y(x))=\max\{\frac{j-r}{a_j-a_r}\mid r<j\leq n,\, h_j=1\}$.
  \end{enumerate}

  \textbf{Case $h>1$}. By Proposition~\ref{pro:reducion_altura_1} there
  exist derivatives 
$H=\frac{\partial^{|\rho|}P}{\partial Y^{\rho}}$, with $|\rho|=h-1$,
such that
the pivot point of $H$ with
respect to $y(x)$ is equal to $1$. Define $s(P;y(x))$ as the minimum of all
those $s(H;y(x)))$. {} If for some derivative $H$, the equation
$H=0$ and its solution $y(x)$ fall in the (RS-N)
case, then $s(P;y(x))$ can be taken as any positive number.{}
\end{definition}

\begin{remark} When $Q$ has ordinate $h>1$, the number $s(P;y(x))$ can
  be described directly in terms of the relative pivot
  points: 
  let $Q=(a,h)$ be the (general) pivot point of $P$ with
  respect to $y(x)$, and $Q_{\rho}(P;y(x))=(a_\rho,h_{\rho})$. 
  Let  $A$ be the set formed by those $3$-tuples $(\rho,i,j)$
  satisfying the following properties: $|\rho|=h$,  $Q_\rho=Q$, $0\leq
  i <j\leq n$, and $h_{\rho'}=h$, where $\rho'=\rho-e_i+e_j${}, using
  the previous notation $e_j = (0,\ldots,0, 1,0,\ldots,0)$ where the $1$ is at
  position $j+1$ (as we need to account for the case $j=0$){}.
  If the set $A$ is empty, we define $s(P;y(x))$ as any
  positive real number. Otherwise, $s(P;y(x))$ is the minimum of
  $\frac{j-i}{a_{\rho'}-a}$, for those $(\rho,i,j)\in A$.
\end{remark}

\begin{remark}\label{re:Zhang-vs-ours}%
  Zhang's paper \cite{Zhang} deals with  
  the case in which $P$
  is a convergent series. The bound given there for the $q$-Gevrey
  order of the 
  solution coincides with the one described here in cases (RS-R) and
  (IS), provided $h_n=1$. In the other cases,  Zhang proves that some
  bound exits but without a detailed control. In particular, our bound
  in case (RS-N) is more accurate because we prove that the solution
  is of $q$-Gevrey order $s$, for any $s>0$. If $h_n=1$, 
  the bound found in~\cite{Zhang}  is described with
  the aid of  the Newton-Adams Polygon
  (see \cite{Adams,Adams-31})
  of the  linearized operator along $y(x)$:
  \begin{displaymath}
    L_y=\sum_{j=0}^{n}\frac{\partial P}{\partial Y_j}[y(x)]\,
    \sigma^{j}\in \mathbb{C}[[x]][\sigma].
  \end{displaymath}
  By Proposition~\ref{pro:reducion_altura_1}, 
  we know that $L_y$ is not identically zero
  if and only if the pivot point of $P$ with respect to $y(x)$ has
  ordinate $1$. The Newton-Adams Polygon $\mathcal{N}_q(L_y)$ of $L_y$
  is defined as follows: for each $0\leq j \leq n$,  let
  $l_j=\ord \frac{\partial P}{\partial Y_j}[y(x)]\in
  \mathbb{N}\cup\{\infty\}$. Notice that $l_j=a_j$ if
  $h_j=1$. Then $\mathcal{N}_q(L_y)$ is the convex
  hull of the set $\{(j,l_j+r)\mid l_j\neq \infty,\,r\geq 0\}$. It is
  easy to check that 
 $s(P;y(x))$ is the reciprocal of the minimum of the positives slopes
 of~$\mathcal{N}_q(L_y)$.  
\end{remark}

\begin{remark}
  The labels (RS-*) and (IS-*) in Definition~\ref{de:q-gevrey_bound}
  correspond 
  to the  singularity type of 
  the linearized operator $L_y$ (regular  or irregular).
  The labels (*-R) or (*-N) denote whether the solution
  $y(x)$ is a regular solution of $P$ (i.e. $h_n=1$) or not.
\end{remark}
\subsection{Reduction to solved form}
In order to prove Theorem~\ref{th:Gevrey-order}, 
we first show (in the paragraphs below) that we may assume that the equation $P[y]=0$ is
in solved form and that the general and
all the
relative pivot points with respect to the variables $Y_j$ are reached at
step $0$.  

Let $y(x)=\sum_{i=0}^{\infty}c_i\,x^{i}\in \mathbb{C}[[x]]$ be a
solution of $P[y]=0$. 
We apply the process 
described in
the proof of Lemma~\ref{le:reduction_to_solved_form} to $P$ and
$y(x)$ in three steps: 
\begin{enumerate}
\item[(a)] Replace $P$ by some of its derivatives
$H$ such that the ordinate of the pivot point of $H$ with
respect to $y(x)$ is equal to $1$ and $s(P;y(x))=s(H;y(x))$.
\item[(b)] Let $N$ be  large enough so that all the relative points
$Q_{j}(H;y(x))$, for $0\leq j \leq n$, have been reached at step 
$N-1$. 
\item[(c)] Let $\gamma=N-1$ and consider $P^*=\mshift{H_{N}}{\gamma}$
  and $y^*(x)=\sum_{i=N}^{\infty}c_i\,x^{i-N+1}$. Then, $P^*[y]=0$ is
  in 
quasi-solved form and $P^*[y^*(x)]=0$. 
\end{enumerate}

If $\bar{y}(x)=\sum_{i=N}^{\infty}c_i\,x^{i}$, then 
the relative pivot points of  $y(x)$
with respect to $H$ are the same as the relative pivot points of
$\bar{y}(x)$ 
with respect to $H_N$. Hence, 
$s(H;y(x))=s(H_N;\bar{y}(x))$. Finally,
{}
the change of variables
\eqref{eq:changeOfVariableMshif} produces the affine transformation $\tau$ on the $(i,j)$ plane on which the Newton polygon is defined; recall that this transformation satisfies 
$\tau(Q_j(H_N;\bar{y}(x)))=Q_j(P^*;y^*(x))$, and moreover, $\tau$ 
restricted to the line of points with ordinate
$1$ is a translation,{} so that $s(H_N;\bar{y}(x))=s(P^*;y^*(x))$. This proves that
$s(P;y(x))=s(P^*;y^*(x))$. 
Moreover, the general and relative pivot points $Q_j(P^*;y^*(x))$ are
reached at step~$0$. 
It is straightforward to prove that if 
$P$ is of  $q$-Gevrey order $t$, then $H$, $H_N$ and $P^*$
are all of $q$-Gevrey order $t$.  
Also $y^*(x)$ and $y(x)$ have the same  $q$-Gevrey order.
This shows that it is
enough to prove Theorem~\ref{th:Gevrey-order} when the
$q$-difference equation $P[y]=0$ is
in quasi-solved form and the relative pivot points
$Q_j(P;y(x))$ are reached at step~$0$. 

Finally, assuming that  $P[y]=0$ is in quasi-solved form,
since $|q|>1$, the set $\Sigma^{+}$ is 
finite. Let $N$ be an integer greater than the maximum of
$\Sigma^{+}$,  $P^{*}=\mshift{(P_{N+1})}{N}$ 
and $y^{*}(x)=\sum_{i=N+1}c_i\,x^{i-N}$. It is clear that
$s(P;y(x))=s(P^*;y^*(x))$, and also that $P^{*}$ and $y^{*}(x)$ are of
the same 
$q$-Gevrey order as $P$ and $y(x)$ respectively. From this we conclude
that we may assume the $q$-difference equation $P[y]=0$ is in solved form.
\subsection{Recursive formula for the coefficients}%
\label{sec:recurs-form-coeff}
Let $y(x)=\sum_{i=1}^{\infty} c_i\,x^{i}$ be a power series solution
of the $q$-difference equation $P[y]=0$, 
where  
$$
P=\sum_{(\alpha,\rho)\in \mathbb{N}\times\mathbb{N}^{n+1}} 
P_{\alpha,\rho}\, x^{\alpha}\,Y^{\rho}
\in \mathbb{C}[[x,\Y]].
$$ 
Assume that it is in solved form and that the general pivot point $Q$
with respect to $y(x)$
and the relative ones $Q_j=(a_j,h_j)$
are all reached at step~$0$. Since the equation
is in solved form, $Q=(0,1)$. Let $r$
be the maximum index $j$, $0\leq j\leq n$, such that
$Q_j=Q$ and let $\Psi(T)$ be the indicial polynomial of $P$ at point
$Q$. From equation~(\ref{eq:ecuacion_recursiva:1}) one has
\begin{equation}
  \label{eq:ec_recursiva_2}
  \Psi(q^{i})\,c_i=-\Coeff(P_i;\,x^{i}\,Y^{\underline{0}}),\quad i\geq 1.
\end{equation}
As usual $P_i=P[c_1x+\cdots+c_{i-1}x^{i-1}+Y]$. We are interested in
computing $\Coeff(P_i;\,x^{i}\,Y^{\underline{0}})$ in terms of
$c_1,c_2,\ldots,c_{i-1}$. 
To this end, we shall consider  formal
series $H^{i}$ in the variables
$T_{\alpha,\rho}$, $C_{j,l}$, $x$ and $\Y$, where
$\alpha\in \mathbb{N}$, 
$\rho=(\rho_0,\ldots,\rho_n)\in  \mathbb{N}^{n+1}$, 
$0\leq j\leq n$ and $1\leq l\leq i-1$, defined as follows
\begin{displaymath}
  H^{i}=\sum_{(\alpha,\rho)\in
    \mathbb{N}^{n+2}}T_{\alpha,\rho}\,\,
   x^{\alpha}
   \prod_{0\leq j\leq n}\left(
    \sum_{l=1}^{i-1}C_{j,l}\,x^{l}+Y_j
      \right)^{\rho_j}.
\end{displaymath}
For $(\beta,\gamma)\in \mathbb{N}\times \mathbb{N}^{n+1}$, let 
$H^{i}_{\beta,\gamma}$ be the coefficient of $x^{\beta}Y^{\gamma}$ in
$H^{i}$. It is a polynomial with coefficients in $\mathbb{N}$ and in the
variables $T_{\alpha,\rho}$ and $C_{j,l}$. 
Denote $L_{i}=H^{i}_{i,\underline{0}}$, i.e. the
coefficient of $x^{i}Y^{\underline{0}}$ in $H^{i}$. 
A simple computation shows that
\begin{displaymath}
  L_{i}=\sum_{(\alpha,\rho,\ud)\in \mathcal{F}_i}
       B^{i}_{\alpha,\rho,\ud}\,\,T_{\alpha,\rho}\,\,
       \prod_{0\leq j\leq n}
        \prod_{1\leq l\leq i-1}C_{j,l}^{d_{j,l}},
\end{displaymath}
where $B^{i}_{\alpha,\rho,\ud}$ are non-negative integers and 
the summantion set $\mathcal{F}_i$ comprises those $(\alpha,\rho,\ud)$,
such that $\alpha\in \mathbb{N}$, $\rho\in \mathbb{N}^{n+1}$, 
$\ud=(d_{j,l})\in \mathbb{N}^{(n+1)(i-1)}$, for $0\leq j\leq n$,
$1\leq l \leq i-1$, for which the following formul{\ae} hold:
\begin{eqnarray}
  \label{eq:condiciones_F_i_A}
  \alpha+\sum_{j,l}l\,d_{j,l}&=&i,\\
  \label{eq:condiciones_F_i_B}
  \sum_{l}d_{j,l}&=&\rho_j,\quad\text{and so, }\,\,\,
\sum_{j,l}d_{j,l}=|\rho|.
\end{eqnarray}

\begin{remark}
Notice that, substituting in $H^{i}$ the variables
$T_{\alpha,\rho}$ for 
$P_{\alpha,\rho}$ and $C_{j,l}$ for $c_{j,l}:=q^{j\,l}\,c_l$, one obtains
$P_i$. Hence,
$$\Coeff(P_i;x^{i}Y^{\underline{0}})=L_{i}(P_{\alpha,\rho},c_{j,l}).$$   
\end{remark}
However, in order to have an optimal control on the $q$-Gevrey 
growth,  we need to be more
precise and use the position of the relative pivot points of $P$ with
respect to $y(x)$, and refine the summation set:
let $\mathcal{F}'_i$ be the subset of $\mathcal{F}_i$ composed by those
$(\alpha,\rho,\ud)$ satisfying the following properties:
\begin{eqnarray}\label{eq:condiciones_F'_i_A}
  &\text{If }&j>r,\,h_j\geq 2,\,\text{ and }l>i/2\,\text{ then
  }d_{j,l}=0.\\\label{eq:condiciones_F'_i_B}
  &\text{If }&j>r,\,h_j=1,\,\text{ and }l>i-a_j\,\text{ then }d_{j,l}=0.
\end{eqnarray}
Let $\mathcal{F}''_i$ be the complement of $\mathcal{F}'_i$ in
$\mathcal{F}_i$ and let $L'_i$ (resp. $L''_i$) be the sum of those
terms in 
$L_i$ corresponding to those $(\alpha,\rho,\ud)$ in $\mathcal{F}'_i$
(resp. in $\mathcal{F}''_i$); obviously $L_i=L'_i+L''_i$.
\begin{lemma}\label{le:coeff_L_prime} The following equality holds:
  $\Coeff(P_i;x^{i}Y^{\underline{0}})=L'_{i}(P_{\alpha,\rho},c_{j,l})$. 
\end{lemma}
\begin{proof}
  Take $l_0$ with $1\leq l_0\leq i-1$, and  consider
  $P_{l_0}=P[\sum_{l=1}^{l_0-1}c_l\,x^{l}+Y]$
  {}(see~(\ref{eq:def_translacion})){}.
  Let
  $\bar{P}_{l_0}$ be
  the series obtained substituting in $P_{l_0}$ the expression
  $\sum_{l=l_0}^{i-1}C_{j,l}\,x^{l}+Y_j$ for the variable $Y_j$,
  $0\leq j\leq n$.
  {}By construction,
  \begin{align*}
   &P_i=P_{l_0}[c_{l_0}\,x^{l_0}+\cdots+c_{i-1}\,x^{i-1}+Y],\\
   &L_i(T_{\alpha,\rho}=P_{\alpha,\rho},C_{j,l}=c_{j,l};1\leq l <l_0)=\coeff(\bar{P}_{l_0};x^{i}Y^{\underline{0}}).
  \end{align*}
  {}%
    Write $P_{l_0}=\sum_{(\alpha,\rho)\in
      \mathbb{N}\times \mathbb{N}^{n+1}}
    (P_{l_0})_{\alpha,\rho}\,x^{\alpha}\,Y^{\rho}$. Expanding
    $\bar{P}_{l_0}$ as a series in the variables $C_{j,l}$, $l_0\leq
    l\leq i-1$, $x$ and $Y_j$, $0\leq j\leq n$,
    let us denote, for $r<j_0\leq n$, by
    $A_{j_0,l_0}$ the sum of terms of $\bar{P}_{l_0}$ in which the
    variable $C_{j_0,l_0}$ appears effectively. In order to
    compute $A_{j_0,l_0}$, it is only necessary to take into account the
    terms of $P_{l_0}$ in which the variable $Y_{j_0}$ appears
    effectively, that is, only consider the sum over the indices
    $(\alpha,\rho)\in \mathcal{C}_{j_0}(P_{l_0})$. Since we are
    assuming that the pivot 
    point $Q_{j_0}=(a_{j_0},h_{j_0})$ of $P$ with respect to $y(x)$
    relative to the variable 
    $Y_{j_0}$ is reached at step $0$, we may assume that the order
    in $x$ of $A_{j_0,l_0}$ is greater than or
    equal to $a_{j_0}+h_{j_0}\,l_0$. If $j_0,l_0$ 
    satisfy the premise of either 
    \eqref{eq:condiciones_F'_i_A} or 
    \eqref{eq:condiciones_F'_i_B}, then 
    $a_{j_0}+h_{j_0}\,l_0>i$ and the variable
    $C_{j_0,l_0}$ does not appear effectively in the coefficient of
    $x^{i}Y^{\underline{0}}$ in $\bar{P}_{l_0}$. From this one infers 
    that   $L''_i(P_{\alpha,\rho},c_{j,l})=0$.    
\end{proof}

 From the definition of $r$, one has
  $\Psi(T)=P_{0,e_0}+P_{0,e_1}\,T+\cdots + P_{0,e_r}\,T^{r}$, with
  $P_{0,e_r}\neq 0$. In particular, $\Psi(T)$ has degree
  $r$. Moreover, since the equation $P[y]=0$ is in solved form, 
   $\Psi(q^{i})\neq 0$, for $i\geq 1$.
 From equation~(\ref{eq:ec_recursiva_2}) and
  Lemma~\ref{le:coeff_L_prime}, the following recursive formula holds
  for all $i\geq 1$:
  \begin{equation}
    \label{eq:recursive_formula}
    c_i=\frac{-1}{\Psi(q^{i})}L'_i(P_{\alpha,\rho};c_{j,l}),
  \end{equation}
 where $c_{j,l}=q^{j\,l}c_{l}$, $1\leq l\leq i-1$ and  $0\leq j\leq
 n$. 

\subsection{A majorant series}
Assume the hypotheses and notations of the previous sub-section and
 that $P$ has $q$-Gevrey
order $t\geq 0$.
Let $s=s(P;y(x))$. 
Consider the  equation in two variables $x$ and $w$:
 \begin{equation}
   \label{eq:ecuacion_mayorante}
   w= \vq^{-\frac{s+t}{2}} \,|c_1|\,x+
     \sum_{(\alpha,\rho)\in \mathcal{C}'}  
     G_{\alpha,\rho} \, x^{\alpha}\,w^{|\rho|},
 \end{equation}
where $G_{\alpha,\rho}=
|P_{\alpha,\rho}|\,
\vq^{-\frac{t}{2}(\alpha+|\rho|)^{2}+k_1(\alpha+|\rho|)+k_2}$, 
$k_1$ and $k_2$ are positive constants to be specified later, 
and $\mathcal{C}'$ is the set $\mathbb{N}\times \mathbb{N}^{n+1}$ without
the points $(0,\underline{0})$, $(1,\underline{0})$ and $(0,e_j)$ for
$0\leq j\leq n$. 
It is straightforward  to prove that the right hand side of
\eqref{eq:ecuacion_mayorante} is a convergent series and that
the equation has a unique power series solution
$w(x)=\sum_{i=1}^{\infty} c'_i\, x^{i}$,
whose coefficients $c'_i$ satisfy the recursive formulae:
\begin{displaymath}
  c'_1=\vq^{-\frac{s+t}{2}} |c_1|,\quad 
  c'_i=L_i(G_{\alpha,\rho};\{c'_{j,l}\}),\quad i\geq 2,
\end{displaymath}
where $c'_{j,l}=c'_{l}$, for $1\leq l\leq i-1$, and $0\leq j\leq n$. 
In particular, $c'_i\geq 0$, for all $i\geq 1$, since the coefficients
of $L_i$ are non-negative. By Puiseux's theorem, the series $w(x)$ is 
convergent. The following lemma finishes the proof of
Theorem~\ref{th:Gevrey-order},
{}
because by using the majorant criterion the series solution
$y(x)=\sum_{i=1}^{\infty}c_i\,x^{i}$ is of $q$-Gevrey order $s+t$.
{}
\begin{lemma} With the above notations, there exist
  positive constants $k_1$ and $k_2$ such that the coefficients $c'_l$
  of the solution of equation~(\ref{eq:ecuacion_mayorante}) satisfy
  \begin{equation}
    \label{eq:majorant_equation_1}
    |c_l|\leq {\vq^{\frac{s+t}{2}l^{2} }}|c'_{l}|,\quad l\geq 1.
  \end{equation}
\end{lemma}
\begin{proof}
The above inequality holds trivially for $l=1$.
  Assume that it holds for $l=1,2,\ldots,i-1$. 
Using equation~\eqref{eq:recursive_formula} and the fact
  that the coefficients of $L_i$ are non-negative, one gets
  \begin{align}\nonumber
    |c_i|\leq&
    \frac{1}{|\Psi(q^{i})|}
  \sum_{(\alpha,\rho,\ud)\in\mathcal{F}'_i}
  B^{i}_{\alpha,\rho,\ud} \,|P_{\alpha,\rho}| \prod_{j,l}(\vq^{jl}|c_l|)^{d_{j,l}}
    \\\label{eq:induccion_c_i}
  \leq&
 \frac{1}{|\Psi(q^{i})|}
 \sum_{(\alpha,\rho,\ud)\in\mathcal{F}'_i}B^{i}_{\alpha,\rho,\ud}\,
 G_{\alpha,\rho}
 \frac{\vq^{\frac{t}{2}(\alpha+|\rho|)^{2}}}{\vq^{k_1(\alpha+|\rho|)+k_2}}
 \prod_{j,l}\left(
 \vq^{jl}\,{\vq^{\frac{s+t}{2}l^{2}}}|c'_{l}| 
\right)^{d_{j,l}}\\\nonumber
=& 
\sum_{(\alpha,\rho,\ud)\in\mathcal{F}'_i} R_i(\alpha,\rho,\ud)\,
 B^{i}_{\alpha,\rho,\ud} \,G_{\alpha,\rho}\,\prod_{j,l}|c'_l|^{d_{j,l}},
 \end{align}
 where the indices $j$ and $l$ are $0\leq j\leq n$ and $1\leq
 l\leq i-1$, and 
 \begin{flalign*}
    R_i(\alpha,\rho,\ud)&=
  \frac{1}{|\Psi(q^{i})|}\,\vq^{r_i(\alpha,\rho,\ud)},
\\
r_i(\alpha,\rho,\ud)&=
\sum_{j,l}(j\,l+\frac{s+t}{2}l^{2})d_{j,l}+
\frac{t}{2}(\alpha+|\rho|)^{2}
-k_1(\alpha+|\rho|)-k_2.
 \end{flalign*}
\textbf{Claim\/} (proved below): there exist positive constants
$k_1$ and 
$k_2$, such that
\begin{equation}
  \label{eq:claim_R_i}
  R_i(\alpha,\rho,\ud)\leq \vq^{\frac{s+t}{2}i^{2}},
 \quad (\alpha,\rho,\ud)\in \mathcal{F}'_{i}.
\end{equation}
Assuming the claim and using equations \eqref{eq:induccion_c_i} and \eqref{eq:claim_R_i}, 
one gets
\begin{displaymath}
  |c_i|\leq \vq^{\frac{s+t}{2}i^{2}}
 \sum_{(\alpha,\rho,\ud)\in\mathcal{F}'_i} 
 B^{i}_{\alpha,\rho,\ud}\,
G_{\alpha,\rho}  \prod_{j,l}|c'_l|^{d_{j,l}}
=\vq^{\frac{s+t}{2}i^{2}} L'_i(G_{\alpha,\rho};\{|c'_l|\}).
\end{displaymath}
Since the coefficients of $L_i$, the elements $G_{\alpha,\rho}$ and
$c'_l$ are all non-negative real numbers, then
$L''_i(G_{\alpha,\rho};\{c'_l \})\geq 0$. Hence, 
$$
L'_i(G_{\alpha,\rho};\{|c'_l|\})\leq 
L'_i(G_{\alpha,\rho};\{c'_l\})+L''_i(G_{\alpha,\rho};\{c'_l\})
= 
L_i(G_{\alpha,\rho};\{c'_l\})=|c'_{i}|,
$$
 which proves the Lemma.
\end{proof}
\begin{proof}[Proof of Claim] Since the degree of $\Psi(T)$ is $r$, $\vq>1$ and
  $\Psi(q^{i})\neq 0$ for $i\geq 1$,  
  there exists a constant $K_2>1$, 
  such that $|q|^{i\,r}\leq K_2\,|\Psi(q^{i})|$, for all $i\geq
  1$. Thus, it is enough to prove that
  there exist $k_1>0$ and $k_2>{}\ln K_2{}/\ln |q|$ such that 
  $r_i(\alpha,\rho,\ud)\leq \frac{s+t}{2}i^{2}+ri$, 
  for all $i\geq 1$ and all
  $(\alpha,\rho,\ud)\in \mathcal{F}'_{i}$. 
  Grouping the terms of $r_i$ and rearranging, we divide the
  inequality above into two parts so that
  it is enough
  to prove the existence of
  positive constants $k_1$ and $k_2$, such that for all
  $(\alpha,\rho,\ud)\in \mathcal{F}'_{i}$ and $i\geq 1$, the following
  inequalities hold:
 \begin{align}
    \label{eq:a_probar_1}
    \frac{s}{2}\sum_{j,l}l^{2}d_{j,l} + \sum_{j,l}j\,l\,d_{j,l}
      &\leq \frac{s}{2}i^2+r\,i+k_2,\\
     \label{eq:a_probar_2}
     \frac{t}{2} \sum_{j,l}  l^{2}\,d_{j,l} +
   \frac{t}{2}(\alpha+|\rho|)^{2}
      &\leq \frac{t}{2}i^2+k_1(\alpha+|\rho|).
 \end{align}
 We first prove the existence{} of $k_2$ such that{}
 inequality~\eqref{eq:a_probar_1} {}holds{}  and then {}we do the same for $k_1$ and equation
~\eqref{eq:a_probar_2}{}.

\noindent\textsl{Proof of inequality~\eqref{eq:a_probar_1}}. 
Call $r'_i(\alpha,\rho,\ud)$ the left hand side
of~\eqref{eq:a_probar_1}. 
Let $\mathcal{F}'_i=F_1\cup F_2$, where $F_1$ is the subset formed by
those $(\alpha,\rho,\ud)$ such that $l>i/2$ implies $d_{j,l}=0$, and
$F_2$ is its complement in $\mathcal{F}'_i$.
 We shall bound
  $r'_i$ in each of $F_1,F_2$ by a polynomial 
  $\bar{r}'(i)=\bar{r}'_2i^{2}+\bar{r}'_1i+\bar{r}'_0$, such that,
 either $\bar{r}'_2<\frac{s}{2}$ or $\bar{r}'_2=\frac{s}{2}$
  and $\bar{r}'_1\leq r$. Adjusting $k_2$ conveniently, one
  gets~\eqref{eq:a_probar_1}. 

  Let $(\alpha,\rho,\ud)\in F_1$.  
This implies  that if
 $d_{j,l}\neq 0$, then $l\leq i/2$. As $j\leq n$, and
 $\sum_{j,l}l\,d_{j,l}\leq i$ (which follows
 form~\eqref{eq:condiciones_F_i_A}),   
we conclude that
 \begin{displaymath}
   r'_i= \frac{s}{2}\sum_{j,l}l^{2}d_{j,l}+\sum_{j,l}j\,l\,d_{j,l}
     \leq  \frac{s\,i}{4}\sum_{j,l}ld_{j,l} + n\sum_{j,l}l\,d_{j,l}
     \leq \frac{s}{4}i^{2}+n\,i=\bar{r}'(i).
 \end{displaymath}
If $s\neq 0$, then $\bar{r}'_2<s/2$. Otherwise, 
$s=0$, and by Definition~\ref{de:q-gevrey_bound},  $r=n$,
hence $\bar{r_1}'\leq r$. This proves that the polynomial
$\bar{r}(i)$ satisfies our requirements.

Let $(\alpha,\rho,\ud)\in F_2$.
There exists a pair $(j_0,l_0)$ such that $l_0>i/2$ and
  $d_{j_0,l_0}\neq 1$.
By inequality~\eqref{eq:condiciones_F_i_A}, this pair is
unique and~$d_{j_0,l_0}=1$. In this case, equation~\eqref{eq:condiciones_F_i_A} reads as 
\begin{equation}\label{eq:eq:condiciones_F_i_A_particular}
  \alpha+\sum_{j,l\neq l_0}l\,d_{j,l}+l_0=i,\quad\text{and in particular
  }\,
 \sum_{j,l\neq l_0}l\,d_{j,l}\leq a,
\end{equation}
where $a=i-l_0 < i/2$. This implies also that for
 $l\neq l_0$ and $d_{j,l}\neq 0$ one has $l\leq a$. Therefore,
\begin{align*}
  r'_i&=
  \frac{s}{2}\left( \sum_{j,l\neq l_0} l^{2}d_{j,l}+l_0^{2}\right)+
  \sum_{j,l\neq l_0} j\,l\,d_{j,l}+j_0\,l_0\\
  &\le \frac{s}{2}\left(a\sum_{j,l\neq l_0} l\,d_{j,l}+(i-a)^{2}\right)+
  n\sum_{j,l\neq l_0}l\,d_{j,l}+j_0(i-a)\\
  &\le \frac{s}{2}(2a^{2}-2a\,i+i^{2})+n\,a+j_0(i-a)\\
  &=(s/2)i^{2}+(j_0-s\,a)i+(s\,a^{2}+n\,a-a\,j_0):=f_i(a).
\end{align*}
For a fixed $i$, the graph of $f_i(a)$ is {}either an upwards
parabola (case $s>0$) or an straight line (case $s=0$){}, 
so its maximum in an interval is reached at its endpoints.
The available range for $a$ depends on $j_0$.
If $j_0\leq r$, then {}there are no additional constrains on $l_0$, so{} $a\in [1,i/2[$, and we take
$\bar{r}'(i)=s/2\,i^{ 2}+r\,i+\bar{r}'_0$. We can chose $\bar{r}'_0$ in
such a way that $\max\{f_i(1),f_i(i/2)\}\leq \bar{r}'(i)$, for all $i\geq
1${} and $0\leq j_0\leq r${}, because{}
$f_i(i/2)\leq (s/4)\,{i}^{2}+n\,i$
and $f_i(1)\leq (s/2)\,{i}^{2}+ r\, i+s+n$.
{}%
If, on the other hand, $j_0>r$,
{} since $l_0>i/2$,  
case~\eqref{eq:condiciones_F'_i_A} does not hold, hence
case~\eqref{eq:condiciones_F'_i_B} holds;
{}
so that $h_{j_0}=1$ and $l_0\leq i-a_{j_0}$, and the
range for $a$ is $[a_{j_0},i/2[$. By definition of $s$, one has
$j_0-s\,a_{j_0}\leq r$ and $s>0$. Consider
$\bar{r}'(i)=s/2\,i^{2}+r\,i+\bar{r}'_0$, where $\bar{r}'_0$ is chosen
in such a way 
that $\max\{f_i(i/2),f_i(a_j);j>r, h_j=1\}\leq \bar{r}'(i)$, for all $i\geq
1$.  Such an $\bar{r}'_0$ exists because
{}as above $f_i(i/2)\leq (s/4)\,{i}^{2}+n\,i$ and{} $f_i(a_j)\leq
(s/2)i^{2}+(j-s\,a_j)i+s\,a_{j}^{2} + na_j$ and $j-sa_j\leq r$ for
those $j$ such that $j>r$ and $h_j=1$. 

\noindent\textsl{Proof of inequality~\eqref{eq:a_probar_2}}. 
For $t=0$, the inequality holds trivially, so we may assume that~$t>0$.
Let
$(\alpha,\rho,\ud)\in \mathcal{F}'_i$. Denote
$d_l=\sum_{j=0}^{n}d_{j,l}$, for $1\leq l\leq i-1$ and let
$l_0$ be the maximum of the indices $l$ such that $d_l\neq 0$.  From 
equations~\eqref{eq:condiciones_F_i_A} and
\eqref{eq:condiciones_F_i_B}, the fact that $l_0\geq 1$ and
$d_{l_0}\geq 1$, one gets:
\begin{displaymath}
  i-|\rho|=\alpha+\sum_{l}l\,d_l-\sum_{l}d_l
    =\alpha+\sum_{l\neq l_0}(l-1)d_{l}+(l_0-1)d_{l_0}\geq \alpha + l_0-1.
\end{displaymath}
From which $i-l_0\geq \alpha +|\rho| -1$.
Taking into account that $\alpha\geq 0$,
equation~\eqref{eq:condiciones_F_i_A},
and the fact that $l_0\geq l$
for any $l$ with $d_l\neq 0$, we conclude that 
\begin{align*}
  i^{2}&=(i-l_0+l_0)^{2}=(i-l_0)^{2}+l_0^{2}+2\,l_0\,(i-l_0)\\
      &\geq (\alpha+|\rho|-1)^{2}+l_0^{2}+2\,l_0
      \left( \sum_{l\neq l_0} l\,d_{l} + l_0(d_{l_0}-1)\right)\\
      &\geq (\alpha+|\rho|-1)^{2}+l_0^{2} + 
        \sum_{l\neq l_0}l^{2}d_l+l_0^{2}(d_{l_0}-1)\\
      &\geq (\alpha+|\rho|)^{2}-2(\alpha+|\rho|)+\sum_{l}l^{2}d_{l}.
\end{align*}
This gives inequality~\eqref{eq:a_probar_2} for
{}$k_1\geq t${} and finishes the proof of Theorem~\ref{th:Gevrey-order}.
\end{proof}

\section{Working example}
Let us consider the $q$-difference equation $P[y]=0$ of order 5 and
  degree 6, where
  \begin{displaymath}
    P=4\,{Y_1}^4
  -9\,{Y_0}^2\,{ Y_1}\,{ Y_2}
 +2\,{ Y_0}^3\,{ Y_2}
-x^3\,{ Y_0}^4\,{ Y_5}^2
+{{x{ Y_0}\,{ Y_2}}\over{q^4}}
-{{x^3\,{ Y_2}}\over{q^4}}
-x^3\,{ Y_0}+x^5,
 \end{displaymath}
and  $q=4$.
Its Newton Polygon is $\mathcal{N}(P)$ in Figure~\ref{picture1}. It
has four vertices $v_0=(3,6), v_1=(0,4),
v_2= (1,2)$, $v_3=(5,0)$ and three sides $L_1,L_2$ and $L_3$ with
respective co-slopes $\gamma_1=-3/2$, $\gamma_2=1/2$ and
$\gamma_3=2$. We apply some steps of Procedure~\ref{Procedure-1} to
$P$. As $P$ is a polynomial, $\mu_{-1}(P)=-\infty$. 

{}
\begin{figure}[h!]
  \begin{tikzpicture}[scale=0.4]
    \draw[step = 0.5,gray!30!white, thin] (0,0) grid (9,6);
    \foreach \i in {1,3,...,7}
    {\draw (\i,0) node[anchor=north] {$\i$};}
    \foreach \i in {1,3,...,5}
    {\draw (0,\i) node[anchor=east] {$\i$};}
    \draw(0,6) -- (0,0) -- (9,0);
    \draw[fill=black] (0,4) circle(3pt);
    \draw[fill=black](1,2) circle(3pt);
    \draw[fill=black] (3,1) circle(3pt);
    \draw[fill=black] (5,0) circle(3pt);
    \draw[fill=black] (3,6) circle(3pt);
    \draw[line width=1pt](9,6) -- (3,6) -- (0,4) -- (1,2) --
    (3,1) -- (5,0) -- (9,0);
    \draw[fill=gray!10!white, opacity=0.3]
    (9,6) -- (3,6) -- (0,4) -- (1,2) --
    (3,1) -- (5,0) -- (9,0);
    \draw (5,3) node{$\mathcal{N}(P)$};
  \end{tikzpicture}\hfil
  \begin{tikzpicture}[scale=0.4]
    \draw[step = 0.5,gray!30!white, thin] (0,0) grid (16,6);
    \foreach \i in {1,3,...,15}
    {\draw (\i,0) node[anchor=north] {$\i$};}
    \foreach \i in {1,3,5}
    {\draw (0,\i) node[anchor=east] {$\i$};}
    \draw(0,6) -- (0,0) -- (16,0);
    \draw[fill=black] (0,4) circle(3pt);
    \draw[fill=black](1,2) circle(3pt);
    \draw[fill=black] (8,0) circle(3pt);
    \draw[fill=black] (3,6) circle(3pt);
    \draw[line width=1pt](16,6) -- (3,6) -- (0,4) -- (1,2) -- (8,0) -- (16,0);
    \draw[fill=gray!10!white, opacity=0.3]
    (16,6) -- (3,6) -- (0,4) -- (1,2) -- (8,0) -- (16,0);
    \foreach \x/\y in {2/3,4/2,6/1}
    {\draw[fill=gray,gray](\x,\y)circle(3pt);};
    \foreach \x/\y in {5/5,7/4,9/3,11/2,13/1,15/0}
    {\draw[fill=gray,gray](\x,\y)circle(3pt);};
    \draw (13,3) node{$\mathcal{N}(P_1)$};
  \end{tikzpicture}\\[5pt]
  \centering
  \begin{tikzpicture}[scale=0.4]
    \draw[step = 0.5,gray!30!white, thin] (0,0) grid (24,6);
    \foreach \i in {1,3,...,23}
    {\draw (\i,0) node[anchor=north] {$\i$};}
    \foreach \i in {1,3,5}
    {\draw (0,\i) node[anchor=east] {$\i$};}
    \draw(0,6) -- (0,0) -- (16,0);
    \draw[fill=black] (0,4) circle(3pt);
    \draw[fill=black](1,2) circle(3pt);
    \draw[fill=black] (3,6) circle(3pt);
    \draw[fill=black] (4.5,1) circle(3pt);
    \draw[fill=black] (9.5,0) circle(3pt);
    \draw[line width=1pt](24,6) -- (3,6) -- (0,4) -- (1,2) -- (4.5,1) -- (9.5,0) -- (24,0);
    \draw[fill=gray!10!white, opacity=0.3]
    (24,6) -- (3,6) -- (0,4) -- (1,2) -- (4.5,1) -- (9.5,0) -- (24,0);
    \foreach \x/\y in
{15/0,13/1,16.5/0,11/2,14.5/1,18/0,9/3,12.5/2,16/1,
19.5/0,7/4,10.5/3,14/2,17.5/1,21/0,6/1,5/5,
8.5/4,12/3,15.5/2,19/1,22.5/0,6.5/5,10/4,13.5/3,
17/2,20.5/1,24/0,4/2,7.5/1,11/0,2/3,5.5/2,9/1,
12.5/0,3.5/3,7/2,10.5/1,14/0}
{\draw[fill=gray,gray](\x,\y)circle(3pt);};
\draw(21,3) node{$\mathcal{N}(P_2)$};
  \end{tikzpicture}  
  \caption{Newton polygons $\mathcal{N}(P)$, $\mathcal{N}(P_1)$ and $\mathcal{N}(P_2)$.}
  \label{picture1}
\end{figure}
{}


In order to find all the possible starting terms $c_0\,x^{\mu_0}$ of
a solution, we need to consider all the vertices and sides of
$\mathcal{N}(P)$ according as formul{\ae}~\eqref{eq:nec_cond_lados}
and~\eqref{eq:nec_cond_vertices}.
For the vertices, we get: 
$\Psi_{(P;v_0)}(T)=-3\,T^{10}$, $\Psi_{(P;v_1)}(T)=T^{2}(T-2)(4T-1)$,
 $\Psi_{(P;v_2)}(T)=T^{2}/q^{4}$, $\Psi_{(P;v_3)}(T)=1$. 
Hence, for $j=0,1,2,3$ the only
satisfiable formula in~\eqref{eq:nec_cond_vertices} is the one
corresponding to 
vertex $v_1$, that is $\Psi_{(P;v_1)}(q^{\mu})=0$ and $-3/2<\mu<1/2$. 
This gives $\mu=-1$ for any nonzero $c$.
For the sides, we get: $\Phi_{(P;\gamma_1)}(c)=
q^{-15}c^{4}(2\,q^{12}-9\,q^{21/2}+4\,q^{9}-c^{2})$,
$\Phi_{(P;\gamma_2)}(c)=c^2/64$, and $\Phi_{(P;\gamma_3)}(c)=(c-1)^2$.
According as~\eqref{eq:nec_cond_lados}, the only possible starting
terms related to the sides are  $\pm 1024 \sqrt{15}\, x^{-3/2}$ and
$x^{2}$. Notice that $L_2$ gives rise to no starting term. 

Following Procedure~\ref{Procedure-1} we choose $x^{2}$, that is
$c_0=1$ and $\mu_0=2$. The polynomial $P_1=P[x^{2}+Y]$ has $33$ terms
that we do not exhibit; its Newton Polygon is $\mathcal{N}(P_1)$
in  Figure~\ref{picture1}. Since $y=0$ is not a solution of $P_1[y]=0$
because $\mathcal{C}(P_1)$ has points on the $OX$-axis, 
we need to perform step $(a.2)$ of Procedure~\ref{Procedure-1},
that is  finding  $\mu>\mu_0=2$ and $c\neq
0$ so that 
$\Psi_{(P_1;\mu)}(c)=0$. Thus, we can only use the vertices $v_2$
and $v'_3$ and side $L'_3$. 

For formula~\eqref{eq:nec_cond_vertices} we get  
$\Psi_{(P_1;v_2)}(T)=\Psi_{(P;v_2)}$ and that $\Psi_{(P_1;v'_3)}(T)$ is a
constant, hence those vertices do not give rise to  subsequent
terms. For side $L'_3$, we get  $\mu_1=7/2$ and 
$\Psi_{(P_1;\mu_1)}(c)=64\,c^2+225792$, so that there are two
possibilities for $c_1$. We choose $c_1=21\sqrt{8}\sqrt{-1}$ and
go on with Procedure~\ref{Procedure-1}. 

Let us consider
$P_2=P_1[c_1x^{\mu_1}+Y]$ whose Newton Polygon is $\mathcal{N}(P_2)$
having a side $L''_3$ of the same co-slope as $L'_3$ and another
$L''_4$ of co-slope $5$. As vertex $v''_3$ gives
$\Psi_{(P_2;v''_3)}(q^{\mu})=q^{2\mu}+ 16384$ which has no real
solutions, it is useless to find $\mu_2$. Hence we must use $L''_4$
which gives $\mu_2=5$ and (after a trivial computation)
$c_2=-88984/65$.

Notice that, after performing the first two steps detailed above and
getting $x^2+21\sqrt{8}\sqrt{-1}\,x^{7/2}$, the fact that $v''_3$
gives rise to a formula which no $\mu>7/2$ can satisfy and that it has
ordinate $1$ implies that, taking $P^{*}=\mshift{P_2}{{\mu_1}}$, the
equation  $P^{*}[y]=0$ is solved form. Therefore, by
Lemma~\ref{le:quasi_solved_form_case} there exists a unique solution
of $P[y]=0$ of the form: 
\begin{displaymath}
  y(x)=x^2+21\sqrt{8}\sqrt{-1}\,x^{7/2}+o(x^{7/2}).
\end{displaymath}
Notice also that as $P^*\in \mathbb{C}[[x^{1/2}]][Y]$,
Lemma~\ref{le:quasi_solved_form_case} guarantees as well that $y(x)\in
\mathbb{C}[[x^{1/2}]]$. 

The pivot point of $P$ with respect to $y(x)$ is 
$Q(y(x);P)=v''_3=(4.5,1)$. {}This means that, from now on, for each transformation $P_i[c_ix^{\mu_i}+Y]$, the supporting line $L_{(P_i;\mu_i)}$ will intersect $\mathcal{N}(P_i)$ on its lowest side, and the topmost vertex of this side will always be that point $(4.5,1)$. Moreover,{}
$Y_2$ is the highest order appearing effectively in it, hence $r=2$ in
Definition~\ref{def:q-gevrey-order}. There being no monomials with
$Y_3$ or $Y_4$ in $P$ we only need consider the pivot point relative
to $Y_5$ which is the point $Q_{e_5}(y(x);P)=(13,1)$ (notice that
$\mathcal{C}_{e_5}(P_2)$ is in the region above and to the right of
the dashed line). Applying  Definition~\ref{def:q-gevrey-order}
formally we would get $s(y(x);P)=\frac{5-2}{13-4.5}=6/17$.

As regards the growth of the coefficients of $y(x)$, we transform it
into a formal power series in order to apply
Theorem~\ref{th:Gevrey-order}. We do this by means of the ramification
$x=t^{2}$. The series $y(t)$ is a solution of a $\bar{q}$-difference
equation $\bar{P}[y]=0$ derived from $P$ with $\bar{q}=q^{1/2}$. The
ramification induces a horizontal homothecy  of ratio $2$ on the cloud
of points of $P$, $P_1$ and $P_2$. Hence
$s(y(t);\bar{P})=\frac{5-2}{2(13-4.5)}= 3/17$ is a bound for the
$\bar{q}$-Gevrey order of~$y(t)$.

\bibliographystyle{amsplain}
\providecommand{\bysame}{\leavevmode\hbox to3em{\hrulefill}\thinspace}
\providecommand{\MR}{\relax\ifhmode\unskip\space\fi MR }
\providecommand{\MRhref}[2]{%
  \href{http://www.ams.org/mathscinet-getitem?mr=#1}{#2}
}
\providecommand{\href}[2]{#2}


\end{document}